%% file: main.tex
\documentclass[letterpaper,11pt,reqno]{amsart}
\usepackage[portrait,margin=1in]{geometry}
\usepackage[T1]{fontenc}
\usepackage[utf8]{inputenc}
\usepackage{microtype}
\input{style}

\begin{document}
\title{Limit Theorems for the Complete Feedback Card Guessing Game}
\author[Dey]{Partha S.~Dey$^\star$}
\author[Manthripragada]{Jaya M. Manthripragada$^\dag$}
\date{\today}
\address{University of Illinois Urbana--Champaign, 1409 W.~Green Street, Urbana, Illinois 61801}
\email{\{$^\star$psdey,$^\dag$jm66\}@illinois.edu}

\begin{abstract}
	We study the optimal score $T_{n,m}$ in a complete feedback card guessing game.
	The deck has $n$ types and $m$ cards of each type.
	Its predictable compensator $S_{n,m}$ is the sum of the conditional success probabilities.
	We prove that $\var(S_{n,m})\leq C((\log\log n)^2+m/\log n)$, uniformly in $m$ for large $n$.
	This bound yields a central limit theorem for every sequence $m=m_n\geq1$ as $n\to\infty$.
	The score variance is asymptotic to its mean.
	We also give uniform Kolmogorov bounds using the conditional Bernoulli moments.
	For $m,n\to\infty$, the excess mean is asymptotic to
	$(\log n)\log\min\{m,\log n\}+(\pi/\sqrt2)\sqrt{m\log n}$.
	The two contributions balance at $m\asymp(\log n)(\log\log n)^2$.
	The variance proof uses reversed occupancy hitting times.
	Discrete hypergeometric tail estimates control their fluctuations in the large-deviation, Gaussian, and saturation regimes.
	For fixed $n$, we also show that the compensator has a Brownian-bridge functional limit as $m\to\infty$.
\end{abstract}


\maketitle
\setcounter{tocdepth}{1}\tableofcontents

%%%%%%%%%%%%%%%%%%%%%%%%%%%%%%%
\section{Introduction}\label{sec:intro}
%%%%%%%%%%%%%%%%%%%%%%%%%%%%%%%

Consider a uniformly shuffled deck of $N=mn$ cards.
There are $n$ types, with $m$ cards of each type.
In each round, the player guesses the top card.
The card is then revealed and removed.
This is \emph{complete feedback}.
The player therefore knows the composition of the remaining deck.
The aim is to maximize the expected number of correct guesses.

Always guessing one type gives exactly $m$ correct guesses.
The optimal strategy is greedy~\cite{DG81} and chooses a type with maximal remaining multiplicity.
We break ties uniformly, using randomness, independent of the shuffle.
Corollary~\ref{cor:ties} allows other non-anticipating tie rules.
Let $T_{n,m}$ denote the resulting score.
We are interested in the distributional behavior of $T_{n,m}$ as $n\to\infty$, with the multiplicity $m=m_n$ arbitrarily growing with $n$.
Card guessing provides a model of sequential prediction with feedback.
It also has two classical statistical applications.

\smallskip\textbf{Selection bias.}
Blackwell and Hodges~\cite{BH57} studied the two-type game in the design of clinical trials.
Suppose that $2m$ subjects receive two treatments, with $m$ subjects in each group.
An experimenter who knows previous assignments may predict the next treatment.
That knowledge can bias the selection of subjects.
The optimal excess score measures the predictability of the design.

\smallskip\textbf{Experiments with feedback.}
Diaconis and Graham~\cite{DG81} studied the game in the assessment of extrasensory-perception experiments.
Feedback changes the null distribution of the score.
Greedy play supplies an extremal benchmark.
Assessing fluctuations requires more than the expected score.

\smallskip

When $m\ge 2$ is fixed, the mean is asymptotic to $\log n\cdot\sum_{i=1}^{m}1/i$ (see~\cite{He_2023}).
When $n\ge2$ is fixed, it is $m$ plus a term of order $\sqrt m$ (see~\cite{DG81}).
Ottolini and Steinerberger~\cite{OS22} identified the correction $\pi\sqrt{m\log n/2}$ for sufficiently large multiplicities.
Theorem~\ref{thm:allmean} connects these regimes for every sequence $m,n\to\infty$.
The crossover occurs at $m\asymp(\log n)(\log\log n)^2$.
In terms of distributional convergence, Ottolini and Tripathi~\cite{OT23} proved a Gaussian CLT for the score for fixed $m$.
The two-type score limits are also known~\cite{KP23}.
We prove Gaussian score fluctuations for every multiplicity sequence when $n\to\infty$.
The variance is always asymptotic to the mean in this case.
Our decomposition separates the predictable compensator from the martingale remainder.
Their covariance need not vanish. However, the compensator has negligible variance on the score scale.

%%%%%%%%%%%%%%%%%%%%%%%%%%%%%%%
\subsection{Main results}\label{ssec:mainres}
%%%%%%%%%%%%%%%%%%%%%%%%%%%%%%%

Our first theorem gives the excess mean above the score $m$ obtained by always guessing one type.

\begin{thm}[Mean for growing multiplicities]\label{thm:allmean}
	For every sequence $n\to\infty$ and $m=m_n\to\infty$,
	\begin{align}\label{eq:allmean}
		\E T_{n,m}-m
		=(1+o(1))\left(\log n\,\log(\min\{m,\log n\})+\frac{\pi}{\sqrt2}\sqrt{m\log n}\right).
	\end{align}
	In particular, we have
	\[
		\E T_{n,m}-m\sim
		\begin{cases}
			(\log n)\,(\log m),      & m\ll (\log n)(\log\log n)^2,                        \\[1mm]
			\left(1+\pi\sqrt{c/2}\right)\log n\,\log\log n,
			                     & m\sim c\,(\log n)(\log\log n)^2,\quad c\in(0,\infty), \\[1mm]
			\pi\sqrt{m\log n/2}, & m\gg (\log n)(\log\log n)^2.
		\end{cases}
	\]
\end{thm}

Let $\pi_{k-1}^*$ be the conditional success probability in round $k$.
Define the compensator $S_{n,m}$ of the score $T_{n,m}$ as
\[
	S_{n,m}:=\sum_{k=1}^{nm}\pi_{k-1}^*,\qquad \mu_{n,m}:=\E T_{n,m}=\E S_{n,m}.
\]
The compensator is a function of the deck and has mean $\mu_{n,m}$.
The difference $M_N=T_{n,m}-S_{n,m}$ is a terminal martingale value, where $N:=nm$. Section~\ref{ssec:martclt} gives the precise filtration.
Throughout,  we use
\[
	\cH_n=\sum_{j=1}^n\frac1j,\text{ and } \mu_{n,m}\geq m-1+\cH_n
\]
by Lemma~\ref{lem:easylb}.
The Gaussian limit for the score holds whenever $n\to\infty$, including every joint limit $m,n\to\infty$.

\begin{thm}[Central limit theorem]\label{thm:clt}
	Let $m=m_n\geq 1$ be arbitrary and write $\mu_{n,m}:=\E T_{n,m}=\E S_{n,m}$.
	As $n\to\infty$,
	\begin{enumeratea}
		\item \label{clt:a}$(S_{n,m}-\mu_{n,m})/\sqrt{\mu_{n,m}}\to 0$ in $L^2$.
		In particular, $S_{n,m}/\mu_{n,m}\to 1$ in $L^2$.
		\item \label{clt:b}$(T_{n,m}-S_{n,m})/\sqrt{\mu_{n,m}} \implies \N (0,1)$.
		\item \label{clt:c} $\var(T_{n,m}) \sim \mu_{n,m}$ and
		\[
			\frac{T_{n,m}-\mu_{n,m}}{\sqrt{\mu_{n,m}}}\implies \N(0,1).
		\]
        \item\label{clt:d} Let $d_K$ denote the Kolmogorov--Smirnov distance.
	There exist absolute constants $C,C',n_0$ such that, for $n\geq n_0$, we have
	\begin{align}\label{eq:uniformrate}
    \begin{split}    
		\sup_{m\geq1}d_K &\left(\frac{T_{n,m}-\E T_{n,m}}{\sqrt{\mu_{n,m}}},\N(0,1)\right)\\ 
        &\leq C\left(\frac{\log(2+\mu_{n,m})}{\sqrt{\mu_{n,m}}}
		+\left(\frac{1+\var(S_{n,m})}{\mu_{n,m}^2}+\frac1{n^2}\right)^{1/5} +\left( \frac{\var(S_{n,m})}{\mu_{n,m}}\right)^{1/3}\right)\\
        &\leq C' \cdot \frac{(1+\log\log n)^{2/3}}{(\log n)^{1/3}}.
            \end{split}
	\end{align}
	\end{enumeratea}
\end{thm}

\begin{rem}\label{rem:rates}
	
	The proof of the rate of convergence uses exact conditional Bernoulli moments and Fan's inequality~\cite{Fan19}, and is given in Section~\ref{ssec:rates}. We expect an error bound of order $(\log n)^{-1/2+o(1)}$, but do not pursue that direction in this article.
\end{rem}

The main estimate underlying the CLT is concentration of the compensator on a scale smaller than the score's standard deviation.

\begin{thm}[Variance of the compensator]\label{thm:var}
	There exist absolute constants $C<\infty$ and $n_0$ such that for all $n\geq n_0$ and all $m\geq1$,
	\[
		\var(S_{n,m})\leq C\bigl(\log^2\log n+{m}/{\log n}\bigr).
	\]
	In particular, $\var(S_{n,m})/\E S_{n,m}\to0$ as $n\to\infty$, uniformly in $m\geq 1$.
\end{thm}

The behavior changes when the number of types is fixed.
The next proposition identifies the compensator limit and recovers the classical first-order mean. Finding the limit for $(T_{n,m}-m)/\sqrt{m}$ is open for fixed $n\ge 3$.

\begin{prop}[Fixed-number-of-types compensator limit]\label{prop:fixedn}
	For each fixed integer $n\geq2$, as $m\to\infty$,
	\[
		\frac{S_{n,m}-m}{\sqrt m}\implies\mathcal B_n:=
		\int_0^1\frac1u\,{\max_{1\le i\leq n}(B_i(u)-\overline B(u))}\,du,
		\quad \overline B=n^{-1}\sum_{i=1}^n B_i,
	\]
	where $B_1,B_{2},\ldots,B_n$ are independent standard Brownian bridges.
	The integral exists in $L^2$.
	The first moments also converge, with
	\[
		\frac{\E T_{n,m}-m}{\sqrt m}\longrightarrow\E\mathcal B_n=\frac\pi2 \E\max_{i\leq n}Z_i,
	\]
	where $Z_1,\ldots,Z_n$ are independent standard normal variables.
\end{prop}

\begin{rem}[Score limits with a fixed number of types]
	For $n=2$, the known score limit~\cite{KP23} is
	\[
		\frac{T_{2,m}-m}{\sqrt m}\implies R,
		\quad\pr(R\leq x)=1-e^{-x^2},\quad x\geq0.
	\]
	For fixed $n\geq3$, the score limit and its joint law with the compensator remain to be identified.
	For $n=1$, both score and compensator equal $m$ deterministically.
\end{rem}

For the reversed deck, let $W_t$ be the largest number of cards of any one type among the final $t$ cards, and define $\tau_j:=\min\{t:W_t=j\}$, $1\le j\le m$.
The birthday approximation in~\cite{He_2023} gives the following fixed-level limit.
We include a direct proof that makes uniformity in the capacity explicit.
Here $(m)_j=m(m-1)\cdots(m-j+1)$.

\begin{prop}[Fixed-level hitting times]\label{prop:exp}
	Fix an integer $j\geq 2$. For $m\geq j$, let $\kappa_{m,j}:=(m)_j/m^j\in[j!/j^j,1]$. Then, as $n\to \infty$,
	\[
		\frac{\kappa_{m,j}\,\tau_j^{\,j}}{j!\, n^{j-1}} \implies \mathrm{Exp}(1),
	\]
	uniformly in $m\geq j$.
\end{prop}

An urn capacity is the number of available cards of its type.
The next coupling compares decks with different common multiplicities at the same reversed time.

\begin{prop}[Capacity coupling]\label{prop:capacitycoupling}
	Fix $n\geq1$. The reversed urn processes for all integer capacities $c\geq1$ can be coupled so that, almost surely, for every $c\leq d$ and every $0\leq t\leq nc$,
	\[
		\sum_{i=1}^k H_i^{(c)}(t)\leq\sum_{i=1}^k H_i^{(d)}(t),\quad1\leq k\leq n,
	\]
	where $H^{(c)}(t)$ is the height vector arranged in decreasing order.
	The total sums are both $t$. Thus $H^{(c)}(t)$ is majorized by $H^{(d)}(t)$. Under this coupling,
	\[
		W_t^{(c)}\leq W_t^{(d)}\text{ for }0\leq t\leq nc,\quad
		\tau_j^{(d)}\leq\tau_j^{(c)}\text{ for }1\leq j\leq c,\quad
		S_{n,c}\leq S_{n,d}.
	\]
\end{prop}

The compensator estimate is uniform in the multiplicity.
The following bounds retain useful information in particular regimes, but are not needed as separate main results.

\begin{rem}[Bounds in special regimes]\label{rem:weaker}
	On the way to Theorem~\ref{thm:var}, we prove two further bounds on $\var(S_{n,m})$, which complement Theorem~\ref{thm:var}.
	\begin{enumeratea}
		\item \emph{Universal bound.} There is a universal constant $C<\infty$ such that for all $n\geq 2$ and $m\geq 1$,
		\begin{align}\label{eq:varSuniv}
			\var(S_{n,m})\leq C m.
		\end{align}
		This estimate holds for every $n\geq2$.
		For large multiplicities, Theorem~\ref{thm:var} improves its $m$ term by a factor of $\log n$.
		Together with $\mu_{n,m}\geq m-1+\cH_n$, it gives
		$\var(S_{n,m})/\mu_{n,m}^2\leq C/\cH_n$.
		That bound proves the compensator law of large numbers and the martingale CLT for every $m$.
		It also proves the score CLT when $m=o(\log n)$.
		Section~\ref{ssec:univ} gives a short proof.

		\item \emph{Logarithmic multiplicity.} For every $\gl>0$ there exist $C_\gl<\infty$ and $n_0(\gl)$ such that for all $n\geq n_0(\gl)$ and all $1\leq m\leq \gl\log n$,
		\begin{align}\label{eq:varSlog}
			\var(S_{n,m})\leq C_\gl\bigl(1+\log m\bigr)^2.
		\end{align}
		This bound retains the dependence on $m$.
		It gives $\var(S_{n,m})=O(1)$ for fixed $m$.
		Section~\ref{ssec:logregime} proves it and uses it in Theorem~\ref{thm:var}.
	\end{enumeratea}
\end{rem}

\subsection{Heuristics and idea of the proof}\label{ssec:heuristic}
The score CLT rests on three observations. Recall that $N=nm$.

\smallskip\noindent i.~\textbf{Martingale noise.}
Each score increment is conditionally Bernoulli and its centered conditional variance is $\pi^*_{k-1}(1-\pi^*_{k-1})$.
The squared probabilities have a negligible sum compared with $\mu_{n,m}$.
Thus the predictable quadratic variation is close to $S_{n,m}$.
Concentration of $S_{n,m}$ then yields a martingale CLT.
To transfer this CLT to the score, we need the stronger bound
$\var(S_{n,m})=o(\mu_{n,m})$.

\smallskip\noindent ii.~\textbf{Time reversal.}
We can read the deck from the bottom and count cards of each type. Let $\widetilde X_i(t)$ be the number of cards of type $i$ among the final $t$ cards. The resulting process has $n$ urns, each of capacity $m$. If $W_t$ is its maximal occupancy at time $t$, then
\[
	S_{n,m}=\sum_{t=1}^{N} {W_t}/{t}
	=m\cH_{N}-\sum_{j=1}^{m}\cH_{\tau_j-1},
	\text{ where } \tau_j:=\min\{t:W_t=j\}.
\]
Minkowski's inequality reduces compensator concentration to estimates for each $\log\tau_j$.

\smallskip\noindent iii.~\textbf{Hitting-time concentration.}
Let $p_j(t):=\pr(\widetilde X_1(t)\geq j)$ and $h_t(j):=\pr(\widetilde X_1(t)=j)$. The union bound and negative association give
\[
	\pr(\tau_j\leq t)\leq np_j(t),\text{ and }
	\pr(\tau_j>t)\leq e^{-np_j(t)}.
\]
The deterministic scale for $\tau_j$ is thus given by $t_j=\min\{t:np_j(t)\geq1\}$.
The exact identity
\[
	\frac{p_j(t+1)-p_j(t)}{p_j(t+1)}
	=\frac{j}{t+1}\phi_j(t+1),\quad
	\phi_j(u):=\frac{h_u(j)}{p_j(u)}
\]
controls the growth of the tail probability.
Bounds on $j\phi_j$ yield concentration on the logarithmic time scale.
There are three relevant estimates.
\begin{itemize}[leftmargin=2em]
	\item For every capacity, $j\phi_j\gtrsim\sqrt j$ up to time $nj$. This gives $\var(\log\tau_j)\lesssim j^{-1}$ and $\var(S_{n,m})\lesssim m$.
	\item For $j\leq\gl\log n$, the threshold mean occupancy is at most $\rho j$, with $\rho=\rho(\gl)<1$. The successive mass ratios above $j$ are bounded above by a constant strictly smaller than one. Hence $j\phi_j\gtrsim_\gl j$ in a suitable window and $\var(\log\tau_j)\lesssim_\gl j^{-2}$. For fixed $j\geq2$, the limiting variance is $\pi^2/(6j^2)$.
	\item For $2\log n<j<m$, interior and saturation estimates give $\var(\log\tau_j)\lesssim(j\log n)^{-1}$. For the terminal level $j=m$, the bound improves to $C/m^2$.
\end{itemize}
Summing the standard deviations gives
\[
	\sqrt{\var(S_{n,m})}
	\lesssim\log\log n+\sqrt{m/\log n}.
\]
This proves the variance theorem.
The mean theorem uses an exact beta-integral representation.
It involves the maximum of independent $\Bin(m,p)$ variables at each fixed $p$.
The region $mp\leq\log n$ gives the logarithmic contribution, and 
binomial moderate deviations give the Gaussian contribution.

\subsection{Related literature}\label{ssec:lit}
Blackwell and Hodges~\cite{BH57} first studied the two-type game.
They computed its optimal mean score
\[
	\E T_{2,m}=m+\frac{2^{2m-1}}{\binom{2m}{m}}-\frac12=m+\frac{\sqrt{\pi m}}{2}+\cO(1).
\]
Diaconis and Graham~\cite{DG81} treated the general $n$ case.
They proved greedy optimality and obtained the mean asymptotics for fixed $n$ and $m\to\infty$.

\begin{thm} [Diaconis and Graham~\cite{DG81}] Let $n\geq 2$ be fixed and $m\to\infty$. Then
	\[
		\E T_{n,m} = m+\frac{\pi}{2}\,\cZ_n\sqrt{m}+o_n(\sqrt{m}),
	\]
	where $\cZ_n$ is the expected value of the maximum of $n$ independent standard Gaussian random variables.
\end{thm}

For fixed $m$ and $n\to\infty$, He and Ottolini~\cite{He_2023} proved the following.

\begin{thm}[He and Ottolini~\cite{He_2023}] Let $m$ be fixed and $n\to\infty$. Then
	\[
		\E T_{n,m} = \cH_m \cH_n + \sum_{j=1}^{m-1} \frac{1}{j} \log \binom{m}{j} + \cO_m\left((\log n)(\log n/n)^{1/m}\right).
	\]
\end{thm}

Ottolini and Steinerberger~\cite{OS22} considered the case where both parameters are large, with $m$ growing at least polylogarithmically in $n$.
\begin{thm}[Ottolini and Steinerberger~\cite{OS22}]
	Let $c,\eps > 0$, and let $n \to \infty$ with $m\ge  c(\log n)^{3+\eps}$. Then
	\[
		\E T_{n,m}= m+ \frac{\pi}{\sqrt{2}}\sqrt{m\log n}+ o\bigl(\sqrt{m \log n}\bigr).
	\]
\end{thm}

Ottolini and Tripathi~\cite{OT23} proved a Gaussian CLT for the score for fixed multiplicity.
Their theorem also supplies an explicit Berry--Esseen rate of $C_m(\log \log n)(\log n)^{-1/2}$.
\begin{thm}[Ottolini and Tripathi~\cite{OT23}]
	Let $m$ be fixed, $\mu_n=\E T_{n,m}$ and $\sigma_n^2=\var(T_{n,m})$.
	As $n\to\infty$,
	\[
		\mu_n\sim\sigma_n^2\sim\cH_m\cH_n,
		\quad (T_{n,m}-\mu_n)/\sigma_n\implies\N(0,1).
	\]
\end{thm}

Kuba and Panholzer~\cite{KP23} give exact distributions, joint limits and moment asymptotics for two-type decks.
In particular, $(T_{2,m}-m)/\sqrt m$ converges to the distribution with c.d.f. $1-e^{-x^2}$ for $x\geq0$. Ottolini and Tripathi~\cite{OT23} also treat unequal decks with bounded maximal multiplicity under additional balance hypotheses.

Different feedback rules lead to different prediction problems.
With no feedback, the cards are never revealed.
With partial feedback, the player learns only whether a guess was correct.
Diaconis, Graham, He and Spiro~\cite{DGHS22} studied partial feedback, and 
Diaconis, Graham and Spiro~\cite{DGS22} developed practical guessing strategies.
Luo~\cite{Luo26} recently settled two conjectures about extremal expectations under partial feedback.
Our results concern only the complete feedback rule.

\subsection{Notation}\label{ssec:nota}
Throughout, $N=nm$ and $[n]=\{1,\ldots,n\}$.
All logarithms are natural.
We use $\cH_k=\sum_{i=1}^{k}1/i$ and $\cH_0=0$.
The symbols $\pr$, $\E$, $\var$ and $\cov$ denote probability, expectation, variance and covariance.
For nonnegative integers $a,b$, $(a)_0=1$ and $(a)_b=a!/(a-b)!$ when $b\leq a$.
We set $(a)_b=0$ when $b>a$.
We also use $\binom ab=0$ outside $0\leq b\leq a$ and $0\log0=0$.

The law $\hypgeom(N,m,t)$ counts type-1 cards in a uniformly chosen sample of size $t$.
The population has $N$ cards, of which $m$ have type 1.
Its probability mass function is
\begin{align}\label{eq:hyppmf}
	h_t(k)=\binom mk\binom{N-m}{t-k}\bigg/\binom Nt.
\end{align}
The binomial convention covers the entire support.
The symbol $\|\cdot\|_r$ denotes the $L^r$ norm.
We write $\implies$ for convergence in distribution.
We use $a\lesssim b$ to mean $a\leq Cb$, $a\asymp b$ for two-sided comparison, and $a\ll b$, if $a/b\to 0$.
Constants $c,C$ may change between lines.
They are absolute unless a dependence is indicated.
The score is $T_{n,m}$.
The hitting times are denoted by $\tau_j$.

\subsection{Organization of the paper}\label{ssec:org}
The paper is organized as follows. Section~\ref{sec:mean} proves the mean theorem through the independent-binomial integral.
Section~\ref{sec:mart} gives the martingale decomposition.
It proves the CLT using Theorem~\ref{thm:var} and establishes the Kolmogorov bound.
Section~\ref{sec:btp} develops the backward urn process and proves the compensator variance bound.
Section~\ref{sec:hitting} proves the fixed-level hitting-time limits and the fixed-type compensator limit.

\section{Mean Score for Growing Multiplicities and Proof of Thoerem~\ref{thm:allmean}}\label{sec:mean}
The following result fills the gap between fixed-multiplicity mean asymptotics and the large-multiplicity result of~\cite{OS22}.
Theorem~\ref{thm:allmean} is proved below using an exact integral representation and separate estimates for its early and Gaussian contributions.

\subsection{The binomial integral}
Let $(Y_i(p))_{1\leq i\leq n}$ be i.i.d.~$\Bin(m,p)$ variables. Define
$Z(p):=\max_iY_i(p)$. Only their distribution at each fixed $p$ is needed below.
\begin{lem}\label{lem:meanintegral}
	For every $n,m\geq1$,
	\begin{align}\label{eq:meanintegral}
		\E T_{n,m}=\int_0^1\frac1p{\E Z(p)}\,dp,
		\quad
		\E T_{n,m}-m=\int_0^1\frac1p(\E Z(p)-mp)\,dp.
	\end{align}
\end{lem}
\begin{proof}
	Conditionally on $S_n(p)=\sum_{i=1}^nY_i(p)=t$, the count vector is multivariate hypergeometric with capacities $m$, hence its maximum has the law of $W_t$.
	Also $S_n(p)\sim\Bin(N,p)$, with $N=nm$. Since $W_0=0$, Tonelli's theorem and the beta integral give
	\[
		\int_0^1\frac{\E Z(p)}p\,dp
		=\sum_{t=1}^N\E W_t \ \binom Nt
		\int_0^1p^{t-1}(1-p)^{N-t}\,dp
		=\sum_{t=1}^N\frac{\E W_t}{t}=\E S_{n,m}.
	\]
	The equality $\E T_{n,m}=\E S_{n,m}$ and $\int_0^1mp\,dp/p=m$ finish the proof.
	Thus independence in~\eqref{eq:meanintegral} is exact after integration.
\end{proof}

\subsection{The early contribution}
For $\log n\geq2$, define $a:=\min(m,\log n)$ and
\[
	I_{n,m}=\int_0^a\frac1\lambda {\E Z(\lambda/m)}\,d\lambda.
\]
\begin{lem}\label{lem:earlymean}
	With absolute implied constants,
	\begin{align}\label{eq:earlymean}
		I_{n,m}=\begin{cases}
			        (\log n)\cH_m+O(m),           & m\leq \log n, \\
			        (\log n)\log\log n+O(\log n), & m\geq \log n.
		        \end{cases}
	\end{align}
\end{lem}
\begin{proof}
	Write $f_j(\lambda)=\pr(\Bin(m,\lambda/m)\geq j)$. Independence gives
	\[
		\E Z(\lambda/m)=\sum_{j=1}^m(1-(1-f_j(\lambda))^n).
	\]
	For $0\leq\lambda\leq j/2$, the single mass at $j$, the inequality
	$\binom mj\geq(m/j)^j$, and $\log(1-x)\geq-2x$ for $0\leq x\leq1/2$ give
	\[
		f_j(\lambda)\geq\binom mj(\lambda/m)^j(1-\lambda/m)^{m-j}
		\geq(\lambda/j)^j e^{-2\lambda}\geq(\lambda/(ej))^j.
	\]
	Therefore, for every integer $1\leq j\leq a$,
	\begin{align*}
		\int_0^a\{1-(1-f_j(\lambda))^n\}\frac{d\lambda}{\lambda}
		 & \geq\int_0^{j/2}\{1-e^{-n(\lambda/(ej))^j}\}\frac{d\lambda}{\lambda} \\
		 & =\frac1j\int_0^{n(2e)^{-j}}\frac{1-e^{-u}}u\,du
		\geq\frac{\log n}{j}-\log(2e)-\frac Cj.
	\end{align*}
	We used $\int_0^z(1-e^{-u})du/u\geq\log z-C$ for $z>0$.
	To see this,
	for $z\leq1$ the right side can be taken nonpositive, and for $z\geq1$ split the integral at one.
	Summing to $m$ when $m\leq \log n$, or to $\lfloor \log n\rfloor$ when $m\geq \log n$, gives the required lower bounds.

	For the upper bounds, the factorial-moment inequality gives
	$f_j(\lambda)\leq\binom mj(\lambda/m)^j\leq(e\lambda/j)^j$.
	Thus the integrand at level $j$ is at most $\min\{1,n(e\lambda/j)^j\}$.
	If $r_j=(j/e)n^{-1/j}$, direct integration gives
	\[
		\int_0^a\min\{1,(\lambda/r_j)^j\}\frac{d\lambda}{\lambda}
		\leq\left(\frac{\log n}{j}+\log\frac{ea}{j}\right)_++\frac1j
	\]
    where $x_+:=\max \{x,0\}$.
    
	For $m\leq \log n$, $r_j<m$ for $j\leq m$, and summing the exact expression gives
	\[
		I_{n,m}\leq (\log n)\cH_m+\sum_{j=1}^m\log(em/j)+\cH_m
		=(\log n)\cH_m+O(m),
	\]
	using $\log(m!)\geq m\log m-m$.
	For $m\geq \log n$, sum the displayed upper bound first over $j\leq A\log n$, with a fixed $A>e^3$.
	The sum of $\log n/j$ is $(\log n)\log\log n+O(\log n)$, and
	$\sum_{j\geq1}(\log(e\log n/j))_+=O(\log n)$.
	For $j>A\log n$, the exact integral is $n(e\log n/j)^j/j\leq e^{-j}/j$.
	Summing this tail proves $I_{n,m}\leq (\log n)\log\log n+O(\log n)$.
\end{proof}

\subsection{The Gaussian contribution}
For $m\geq \log n$ define
\[
	J_{n,m}:=\int_{\log n/m}^1\frac{\E Z(p)-mp}{p}\,dp.
\]
\begin{lem}\label{lem:bulkmean}
	As $n\to\infty$, if $m/\log n\to\infty$, then
	\begin{align}\label{eq:bulkmean}
		J_{n,m}\sim\frac{\pi}{\sqrt2}\sqrt{m\log n}.
	\end{align}
	If $1\leq m/\log n\leq R$ for a fixed $R$, then $0\leq J_{n,m}\leq C_R \log n$.
\end{lem}
\begin{proof}
	Let $q=mp(1-p)$. The centered binomial satisfies the Bernstein moment-generating-function bound~\cite[Chapter 2]{BLM13}
	\[
		\log\E e^{t(Y_i(p)-mp)}\leq\frac{qt^2}{2(1-t/3)},\quad0<t<3.
	\]

	For $0<t<3$, Jensen's inequality gives
	\[
		\E Z(p)-mp\leq\frac{\log n}{t}+\frac{qt}{2(1-t/3)}.
	\]
	When $q>0$, take
	$t=\sqrt{2\log n/q}/\{1+\sqrt{2\log n/q}/3\}$.
	When $q=0$, the maximum is deterministic.
	Thus

	\begin{align}\label{eq:binmaxupper}
		0\leq\E Z(p)-mp\leq\sqrt{2q\log n}+\log n/3.
	\end{align}
	Consequently, with $r=m/\log n\geq1$,
	\[
		J_{n,m}\leq\sqrt{2m\log n}\int_{\log n/m}^1\sqrt{\frac{1-p}{p}}\,dp
		+\frac{\log n}{3}\log r.
	\]
	This proves the bounded-$r$ assertion.
	If $r\to\infty$, then
	$(\log n)\log r=o(\sqrt{m\log n})$ and the integral tends to $\pi/2$, proving the upper half of~\eqref{eq:bulkmean}.

	For the matching lower bound, fix $\delta\in(0,1/2)$ and $\eps\in(0,1)$.
	Uniformly for $p\in[\delta,1-\delta]$, define
	$d:=(1-\eps)\sqrt{2q\log n}$ and $k=\lceil mp+d\rceil$.
	Since $\log n=o(m)$, $d=o(m)$ and the integer block
	$k\leq\ell\leq k+\lfloor\sqrt{q/\log n}\rfloor$ lies in the binomial support for all sufficiently large $n$.
	Stirling's bounds give, uniformly in this block,
	\[
		\pr(Y_i(p)=\ell)\geq\frac{c_\delta}{\sqrt q}
		\exp(-m D(\ell/m\Vert p)-C_\delta),
	\]
	where $D(x\Vert p)=x\log(x/p)+(1-x)\log((1-x)/(1-p))$.
	Taylor expansion on the compact interval gives
	\[
		mD(\ell/m\Vert p)=(1-\eps)^2\log n+o(\log n).
	\]
	This holds uniformly in $p$ and $\ell$.
	The cubic remainder is $O_\delta((\log n)\sqrt{\log n/m})=o(\log n)$.
	Rounding and the block width contribute $O_\delta(1)+o(1)$ to the quadratic term.
	Summing the block, whose length is of order $\sqrt{q/\log n}$, yields
	\[
		\pr(Y_i(p)\geq mp+d)
		\geq\frac{c_\delta}{\sqrt{\log n}}\exp\{-(1-\eps)^2\log n-o(\log n)\}.
	\]
	Its product with $n$ tends to infinity uniformly in $p$.
	Independence therefore implies $\pr(Z(p)\geq mp+d)=1-o(1)$ uniformly.
	Moreover, $\E(mp-Z(p))_+\leq\E|Y_1(p)-mp|\leq\sqrt q$, so
	\[
		\E Z(p)-mp\geq(1-\eps)\sqrt{2q\log n}(1-o(1))-\sqrt q.
	\]
	Integrating over $[\delta,1-\delta]$, which is contained in $[\log n/m,1]$ eventually, gives
	\[
		\liminf\frac{J_{n,m}}{\sqrt{m\log n}}
		\geq(1-\eps)\sqrt2\int_\delta^{1-\delta}\sqrt{\frac{1-p}{p}}\,dp.
	\]
	Letting $\eps,\delta\downarrow0$ proves the lower half.
\end{proof}

\subsection{Proof of Theorem~\ref{thm:allmean}}
	If $m\leq \log n$, Lemmas~\ref{lem:meanintegral} and~\ref{lem:earlymean} give
	\[
		\E T_{n,m}-m=(\log n)\cH_m+O(m)\sim (\log n)\log m,
	\]
	because $m\to\infty$ and $m\leq \log n$. The term $\sqrt{m\log n}$ is negligible compared with $(\log n)(\log m)$.
	If $m\geq \log n$, the exact decomposition is
	\[
		\E T_{n,m}-m=(I_{n,m}-\log n)+J_{n,m}.
	\]
	The early term is $(\log n)\log\log n+O(\log n)$. If $m/\log n\to\infty$, Lemma~\ref{lem:bulkmean} gives the desired additive first-order formula.
	If $m/\log n$ is bounded, both $J_{n,m}$ and $\sqrt{m\log n}$ are $O(\log n)$ and hence negligible compared with $(\log n)\log\log n$.
	Every sequence of ratios has a further subsequence on which $m/\log n$ is bounded or tends to infinity.
	The subsequence criterion therefore covers arbitrary sequences.
	The three regime statements follow by comparing the two terms in~\eqref{eq:allmean}.\qed

\section{Martingale Analysis and Proof of Theorem~\ref{thm:clt}}\label{sec:mart}

\subsection{Setup}\label{ssec:setup}
Let $X_k$ be the type of the $k$-th revealed card.
Let $Y_k$ be the corresponding guess.
Write $\vX=(X_1,X_{2},\ldots,X_N)$ and $\vY=(Y_1,Y_{2},\ldots,Y_N)$.
After $k$ rounds, define
\[
	m_k(i):=\sum_{t=1}^{k}\ind_{\{X_t=i\}},\quad
	\pi_k(i):=\frac{m-m_k(i)}{N-k},\quad 0\leq k<N.
\]
Here $m_0(i)=0$.
Given the first $k$ cards, the next type has probability vector $\pi_k$.
The greedy guess is uniform on $A_k=\argmax_i\pi_k(i)$.
Its auxiliary random choice is independent of the deck.
Thus $X_{k+1}$ and $Y_{k+1}$ are conditionally independent given the past cards and guesses.

Define 
$$
\s_k:=\ind_{\{X_k=Y_k\}}\text{ for }k=1,2,\ldots,N.
$$
The conditional success probability is
\[
	\pi_k^*=\E(\s_{k+1}\mid\vX_{1:k},\vY_{1:k})
	=\max_i\pi_k(i)=\frac{m-\min_i m_k(i)}{N-k}.
\]
The score and compensator are defined as
\begin{align}
	T_{n,m} & =\sum_{k=1}^{N}\s_k,\label{def:Tnm}        \\
	S_{n,m} & =\sum_{k=1}^{N}\pi_{k-1}^*.\label{def:Snm}
\end{align}
Indeed, every type in $A_k$ has the same probability $\max_i\pi_k(i)$.
In particular, $S_{n,m}$ depends only on the deck.
It satisfies $\E S_{n,m}=\E T_{n,m}$.
Conditioning on the entire deck gives a different quantity.
Under uniform tie-breaking,
\[
	\E(T_{n,m}\mid\vX)
	=\sum_{k=1}^{N}{\ind_{\{X_k\in A_{k-1}\}}}/{|A_{k-1}|}.
\]
This quantity need not equal $S_{n,m}$.
The distinction matters because future cards affect the compensator.
The following lemmas are useful.
\begin{lem}\label{lem:easylb}
	For every $n,m\geq1$,
	$\E S_{n,m}\geq\cH_n+m-1$.
\end{lem}
\begin{proof}
	Since $W_t\geq\lceil t/n\rceil$ deterministically,~\eqref{eq:SW} gives
    
	\smallskip$\qquad\qquad\displaystyle
		\E S_{n,m}\geq\sum_{t=1}^{nm}\frac{\lceil t/n\rceil}{t}
		=\sum_{i=1}^m i\sum_{t=(i-1)n+1}^{in}\frac1t
		\geq\cH_n+\sum_{i=2}^m i\frac n{in}
		=\cH_n+m-1.
	$
\end{proof}

\begin{lem}\label{lem:pisquared}
	As $n\to\infty$, uniformly in $m\geq1$,
	\[
		\frac1{\E S_{n,m}}\sum_{k=1}^{N}\E(\abs{\pi_{k-1}^*}^2)\longrightarrow0.
	\]
\end{lem}
The proof follows from Lemma~\ref{lem:Bsquared} and the fact that $\E S_{n,m}\geq\max(m,\cH_n)$.
%\begin{proof}
%Lemma~\ref{lem:Bsquared} and $\mu=\E S_{n,m}\geq\max(m,\cH_n)$ give
%\begin{align}\label{eq:Qrate}
%\frac1\mu\sum_{k=1}^N\E(\pi_{k-1}^*)^2
%\leq C\left(\frac1\mu+\frac1n\right)
%\leq C\left(\frac1{\cH_n}+\frac1n\right).
%\end{align}
%The last expression tends to zero uniformly in $m$.
%\end{proof}

\subsection{The martingale central limit theorem}\label{ssec:martclt}
Let $\cF_k=\sigma(X_i,Y_i:1\leq i\leq k)$ and let $\cF_0$ be trivial.
The conditional law is $\s_k\mid\cF_{k-1}\sim\Ber(\pi_{k-1}^*)$.
Define
\[
	M_t:= \sum_{k=1}^{t}\left( \s_k-\pi_{k-1}^*\right),\quad t\in \{0,1,\ldots,N\}.
\]
The process $(M_t,\cF_t)_{0\leq t\leq N}$ is a mean-zero martingale because
$\E(M_t-M_{t-1}\mid\cF_{t-1})=0$.
Its terminal value is
\[
	M_N=T_{n,m}-S_{n,m}.
\]
For the variance of the martingale, we have
\begin{align}
	\var(M_N)
	 & = \sum_{k=1}^{N}\E\bigl((M_k - M_{k-1})^2\bigr)
	=\sum_{k=1}^{N} \E\left( \pi_{k-1}^* (1 - \pi_{k-1}^*)\right)
	= \E S_{n,m} - \sum_{k=1}^{N} \E(\abs{\pi_{k-1}^*}^2).\label{eq:varM}
\end{align}
Moreover, since $T_{n,m}=M_N+S_{n,m}$, the Cauchy--Schwarz inequality gives
\begin{align}
	\bigl|\var(T_{n,m})-\var(M_N)-\var(S_{n,m})\bigr|=2\abs{\cov(M_N,S_{n,m})}\leq 2\sqrt{\var(M_N)\var(S_{n,m})}.\label{eq:varT}
\end{align}
The covariance need not vanish, since $S_{n,m}$ depends on the future deck.
However, the estimate~\eqref{eq:varT} is sufficient for the main theorem.
We first prove the Gaussian limit for $M_N$ under a relative variance condition on $S_{n,m}$.
The condition is then verified in Section~\ref{sec:btp}.

\begin{prop}\label{prop:cltcond}
	Suppose $\var(S_{n,m})/(\E S_{n,m})^2 \to 0$ as $n\to\infty$. Then, the following Gaussian Central Limit Theorem holds,
	$$
		\frac{T_{n,m}-S_{n,m}}{\sqrt{\E S_{n,m}}} \implies \N (0,1) \text{ as } n \to \infty.
	$$
\end{prop}
\begin{proof}
	Recall that $\mu_{n,m}:=\E S_{n,m}$ and
	\begin{align*}
		V_N := \sum_{k=1}^{N}\E( (M_{k}-M_{k-1})^2\mid \cF_{k-1}) =\sum_{k=1}^{N} \pi_{k-1}^* (1 - \pi_{k-1}^*) = S_{n,m} - \sum_{k=1}^{N} \abs{\pi_{k-1}^*}^2.
	\end{align*}
	By Lemmas~\ref{lem:easylb} and~\ref{lem:pisquared}, we know that $\mu_{n,m}\to\infty$ and $\E(\abs{V_N-S_{n,m}})/\mu_{n,m}\to 0$ as $n\to\infty$.
	Now,
	\begin{align*}
		\E(\abs{V_N/\mu_{n,m}-1})
		\leq \frac{1}{\mu_{n,m}}\sqrt{\var(S_{n,m})} +\frac{1}{\mu_{n,m}}\E(\abs{V_N-S_{n,m}}).
	\end{align*}
	Thus, under the condition that $\var(S_{n,m})/\mu_{n,m}^2 \to 0$ as $n\to\infty$, we can say that
	\begin{align}\label{prob_conv_1}
		\lim_{n\to\infty}\E(\abs{V_N/\mu_{n,m}-1}) = 0,
	\end{align}
	so that $V_N/\mu_{n,m}\to 1$ in probability.
	Since $|M_k-M_{k-1}|\leq1$ and $\mu_{n,m}\to\infty$, for each $\eps>0$ the normalized conditional Lindeberg sum is exactly zero for all sufficiently large $n$:
	\[
		\frac1{\mu_{n,m}}\sum_{k=1}^N\E\left((M_k-M_{k-1})^2
		\ind_{\{|M_k-M_{k-1}|>\eps\sqrt{\mu_{n,m}}\}}\mid\cF_{k-1}\right)=0.
	\]
	Let $\sigma_N^2=\var(M_N)=\E V_N$.
	Lemma~\ref{lem:pisquared} gives $\sigma_N^2/\mu_{n,m}\to1$.
	Since $V_N/\sigma_N^2\to1$ in probability and the conditional Lindeberg condition holds, the martingale central limit theorem with deterministic limiting variance~\cite{HH80} gives $M_N/\sigma_N\implies\N(0,1)$. Slutsky's lemma completes the proof.
\end{proof}

\subsection{Proof of Theorem~\ref{thm:clt}.\ref{clt:a}--\ref{clt:c} assuming Theorem~\ref{thm:var}}
	Let $\mu_{n,m}=\E T_{n,m}=\E S_{n,m}$. By Lemma~\ref{lem:easylb}, $\mu_{n,m}\geq m-1+\cH_n$, so that $\mu_{n,m}\geq\cH_n\vee m$. By Theorem~\ref{thm:var}, for $n\geq n_0$,
	\[
		\E\biggl(\biggl(\frac{S_{n,m}-\mu_{n,m}}{\sqrt{\mu_{n,m}}}\biggr)^2\biggr)=\frac{\var(S_{n,m})}{\mu_{n,m}}\leq \frac{C\log^2\log n}{\cH_n}+\frac{C\,m}{\mu_{n,m}\log n}\leq \frac{C\log^2\log n}{\cH_n}+\frac{C}{\log n}\to 0,
	\]
	uniformly in $m$. This proves (a), since $\mu_{n,m}\to\infty$.
	In particular $\var(S_{n,m})/\mu_{n,m}^2\to0$, so that Proposition~\ref{prop:cltcond} applies and gives (b). By (a), (b) and Slutsky's lemma,
	\[
		\frac{T_{n,m}-\mu_{n,m}}{\sqrt{\mu_{n,m}}}=\frac{T_{n,m}-S_{n,m}}{\sqrt{\mu_{n,m}}}+\frac{S_{n,m}-\mu_{n,m}}{\sqrt{\mu_{n,m}}}\implies \N(0,1).
	\]
	Equation~\eqref{eq:varM} and Lemma~\ref{lem:pisquared} give $\var(M_N)=\mu_{n,m}(1+o(1))$.
	Also $\var(S_{n,m})=o(\mu_{n,m})$.
	Thus~\eqref{eq:varT} yields
	\[
		\var(T_{n,m})=\mu_{n,m}(1+o(1))+o(\mu_{n,m})+\cO\bigl(\sqrt{\mu_{n,m}\cdot o(\mu_{n,m})}\bigr)=\mu_{n,m}(1+o(1)).
	\]
	The variance equivalence and Slutsky's lemma replace $\sqrt{\mu_{n,m}}$ by $\sqrt{\var(T_{n,m})}$ in the preceding convergence, proving (c). \qed

\begin{cor}[Arbitrary greedy tie-breaking]\label{cor:ties}
	Theorem~\ref{thm:clt}  holds for every nonanticipating strategy that always guesses a type of maximal remaining multiplicity, using any auxiliary randomness independent of the shuffle.
\end{cor}
\begin{proof}
	Whatever the selected maximal type, the conditional success probability is $\pi_{k-1}^*$. Thus the compensator, predictable quadratic variation and bounded-increment estimates used in the CLT and the rate bound are unchanged.
\end{proof}

\begin{rem}
	If one uses the universal bound~\eqref{eq:varSuniv} instead of Theorem~\ref{thm:var}, the same argument gives
	\[
		\frac{\var(S_{n,m})}{(\E S_{n,m})^2}\leq \frac{Cm}{(m-1+\cH_n)^2}\leq \frac{C}{\cH_n}\to 0,
	\]
	We used $(m-1+\cH_n)^2\geq m\cH_n$, since $\cH_n\geq1$.
	This proves $S_{n,m}/\mu_{n,m}\to1$ in $L^2$ and Theorem~\ref{thm:clt}\ref{clt:b} for arbitrary $m$.
	If $m=o(\log n)$, it also gives
	$\var(S_{n,m})/\mu_{n,m}\leq Cm/(m-1+\cH_n)\to0$.
	The score CLT follows in that range.
	Similarly,~\eqref{eq:varSlog} gives Theorem~\ref{thm:clt}\ref{clt:c} when $m\leq\gl\log n$.
\end{rem}

\subsection{The Kolmogorov bound}\label{ssec:rates}
For random variables $Z$ and $Z'$, let 
\[
    d_K(Z,Z'):=\sup_x\abs{\pr(Z\leq x)-\pr(Z'\leq x)}.
\]

\begin{lem}[Squared success probabilities]\label{lem:Bsquared}
	Define $B_N:=\sum_{k=1}^N\abs{\pi_{k-1}^*}^2$.
	For all sufficiently large $n$ and every $m\geq1$,
	\begin{align}\label{eq:BL2}
		\|B_N\|_2\leq C\left(1+\frac mn\right).
	\end{align}
	In particular, if $\mu_{n,m}=\E S_{n,m}$, then
	\[
		\frac{\E B_N}{\mu_{n,m}}\leq C\left(\frac{1}{\mu_{n,m}}+\frac1n\right),\quad
		\frac{\E B_N^2}{\mu_{n,m}^2}\leq C\left(\frac1{\mu_{n,m}^2}+\frac1{n^2}\right).
	\]
\end{lem}
\begin{proof}
	The multinomial comparison in Corollary~\ref{cor:capacitycomparisons} gives
	\[
		\pr(W_t\geq k)\leq \min\left(1,n\left(\frac{et}{nk}\right)^k\right).
	\]
	We first prove two fourth-moment bounds.
	If $t\leq\sqrt n$ and $k\geq4$, then, for $n\geq4e^2$,
	\[
		\pr(W_t\geq k)\leq n(e/\sqrt n)^k
		\leq\frac{e^4}{n}\,2^{-(k-4)}.
	\]
	The identity
	$\E W_t^4=\sum_{k\geq1}(k^4-(k-1)^4)\pr(W_t\geq k)$
	therefore gives $\|W_t\|_4\leq C$ in this range.
	For any $t$, let $a_t=\lceil2et/n+4\log n\rceil$.
	For integers $k\geq a_t$, we have
	\[
		\pr(W_t\geq k)\leq n2^{-k}
		\leq e^{-(\log2-1/4)k}.
	\]
	Summing the same moment identity gives
	$\|W_t\|_4\leq C(t/n+\log n)$.
	Now, we use time reversal and Minkowski's inequality.
	They give
	\begin{align*}
		\|B_N\|_2
	 \leq\sum_{t=1}^{nm}\frac{\|W_t\|_4^2}{t^2}                        
		 & \leq C\sum_{t\leq\sqrt n}\frac1{t^2}
		+C\sum_{\sqrt n<t\leq nm}\left(\frac1{n^2}+\frac{(\log n)^2}{t^2}\right) \\
		 & \leq C\left(1+\frac mn+\frac{(\log n)^2}{\sqrt n}\right)
		\leq C\left(1+\frac mn\right).
	\end{align*}
	Since $\mu_{n,m}\geq m$, division by $\mu_{n,m}$ proves the remaining assertions.
	No independence between different times is used.
\end{proof}

\begin{proof}[Proof of Theorem~\ref{thm:clt}.\ref{clt:d}]
	Write $S=S_{n,m}$, $M=T_{n,m}-S$, $\mu=\E S$ and $v=\var S$.
	Let $D_k=\s_k-\pi_{k-1}^*$ and abbreviate $p_k=\pi_{k-1}^*$.
	The conditional Bernoulli law gives, for every $r\geq2$,
	\begin{align}\label{eq:Bernmoments}
		\E(|D_k|^r\mid\cF_{k-1})
		 & =p_k(1-p_k)((1-p_k)^{r-1}+p_k^{r-1})
		  \leq p_k(1-p_k)
		=\E(D_k^2\mid\cF_{k-1}).\notag
	\end{align}
	In particular, $\xi_k=D_k/\sqrt\mu$ satisfies
	\[
		\E(|\xi_k|^3\mid\cF_{k-1})
		\leq\mu^{-1/2}\E(\xi_k^2\mid\cF_{k-1}),\quad
		\max_k|\xi_k|\leq\mu^{-1/2}.
	\]
	For large $n$, $\mu^{-1/2}\leq1/2$ uniformly in $m$.
	Thus condition (A1) of Fan~\cite{Fan19} holds with $\rho=1$ and $\epsilon=\mu^{-1/2}$.
	The normalized predictable quadratic variation is
	\[
		Q=\sum_k\E(\xi_k^2\mid\cF_{k-1})=\frac{S-B_N}{\mu}.
	\]
	Lemma~\ref{lem:Bsquared} gives
	\[
		\E|Q-1|^2\leq\frac{2v+2\E B_N^2}{\mu^2}
		\leq C\left(\frac{v+1}{\mu^2}+\frac1{n^2}\right).
	\]
	Fan's Theorem 2.2, with moment parameter $p=2$, gives
	\[
		d_K(M/\sqrt\mu,\N(0,1))
		\leq C\left(\epsilon|\log\epsilon|
		+\left(\E|Q-1|^2+\E\max_k|\xi_k|^4\right)^{1/5}\right).
	\]
	Since $\E\max_k|\xi_k|^4\leq\mu^{-2}$, we obtain
	\[
		d_K(M/\sqrt\mu,\N(0,1))
		\leq C\left(\frac{\log(2+\mu)}{\sqrt\mu}
		+\left(\frac{v+1}{\mu^2}+\frac1{n^2}\right)^{1/5}\right).
	\]

	For the score, first use normalization by $\sqrt\mu$.
	Define $Z:=M/\sqrt\mu$ and $U=(S-\mu)/\sqrt\mu$.
	For every $\delta>0$,
	\begin{align*}
		\pr(Z\leq x-\delta)-\pr(|U|>\delta)
		 & \leq\pr(Z+U\leq x)                       
		 \leq\pr(Z\leq x+\delta)+\pr(|U|>\delta).
	\end{align*}
	The bounded normal density and Chebyshev's inequality imply
	\[
		d_K((T_{n,m}-\mu)/\sqrt\mu,\N(0,1))
		\leq d_K(Z,\N(0,1))+C\delta+\frac{v}{\mu\delta^2}.
	\]
	Take $\delta=(v/\mu)^{1/3}$ if $v>0$.
	If $v=0$, then $U=0$ almost surely.
	This proves the asserted bound with normalization by $\sqrt\mu$.

	To replace $\sqrt\mu$ by $\sigma:=\sqrt{\var(T_{n,m})}$, define $\sigma_M^2:=\var M=\mu-\E B_N$.
	The reverse triangle inequality gives $|\sigma-\sigma_M|\leq\sqrt v$.
	Consequently,
	\[
		\left|\frac{\sigma}{\sqrt\mu}-1\right|
		\leq C\left(\sqrt{\frac v\mu}+\frac1\mu+\frac1n\right).
	\]
	Here $\sigma_M^2\geq\mu/2$ for large $n$, by Lemma~\ref{lem:Bsquared}.
	Theorem~\ref{thm:var} also gives $v/\mu\to0$ uniformly.
	Normal scale comparison costs at most the last displayed bound.
	For $0\leq v/\mu\leq1$, its square-root term is at most $(v/\mu)^{1/3}$.
	The other terms are absorbed by $\{(v+1)/\mu^2+n^{-2}\}^{1/5}$.
	Thus the score bound follows without any independence assumption on $M$ and $S$.
	Finally, Theorem~\ref{thm:var} and $\mu\geq\max(m,\cH_n)$ imply
	\[
		\frac{v+1}{\mu^2}
		\leq C\frac{(1+\log\log n)^2}{(\log n)^2},\quad
		\frac v\mu\leq C\frac{(1+\log\log n)^2}{\log n}.
	\]
	For the first inequality, use $m/\mu^2\leq1/\mu\leq1/\log n$.
	Also $\log(2+\mu)/\sqrt\mu\leq C\log\log n/\sqrt{\log n}$ for large $n$.
	Both terms in the martingale estimate are $o((1+\log\log n)^{2/3}/(\log n)^{1/3})$, uniformly in $m$.
	This proves~\eqref{eq:uniformrate}.
\end{proof}

%%%%%%%%%%%%%%%%%%%%%%%%%%%%%%%
\section{Backward Time Process and Proof of Theorem~\ref{thm:var}}\label{sec:btp}
%%%%%%%%%%%%%%%%%%%%%%%%%%%%%%%

%%%%%%%%%%%%%%%%%%%%%%%%%%%%%%%
\subsection{The urn model}\label{ssec:urn}
%%%%%%%%%%%%%%%%%%%%%%%%%%%%%%%
In this section, we read the deck from the bottom.
Let $\widetilde X_i(t)$ be the number of cards of type $i$ among its final $t$ cards.
Equivalently, $\widetilde X_i(t)=m-m_{N-t}(i)$.
The next reversed card enters urn $i$ with probability
\[
	\frac{m-\widetilde X_i(t)}{N-t},\quad 0\leq t<N.
\]
The process starts from the zero vector.
At time $t$, its coordinates sum to $t$ and are at most $m$.
Uniform shuffling gives the multivariate hypergeometric law
\begin{align}\label{eq:mvhyp}
	\pr(\widetilde\vX(t)=\vx)
	=\binom Nt^{-1}\prod_{i=1}^{n}\binom m{x_i},\quad
	\vx\in\{0,\ldots,m\}^n,\quad\sum_i x_i=t.
\end{align}
Each coordinate has mass function $h_t$ from~\eqref{eq:hyppmf}.
Another representation uses independent $Y_i\sim\Bin(m,p)$, for any $0<p<1$.
Conditioning their sum to equal $t$ gives exactly the law in~\eqref{eq:mvhyp}.

Define the maximal occupancy
\[
	W_t:=\max_{i\in[n]}\widetilde X_i(t),\quad 0\leq t\leq N,
\]
so that $\pi^*_{k}=W_{N-k}/(N-k)$ for $0\leq k\leq N-1$, and by~\eqref{def:Snm}
\begin{align}\label{eq:SW}
	S_{n,m}=\sum_{t=1}^N \frac{W_t}{t}.
\end{align}

Let $\tau_j$ represent the first time $t$ for which $W_t = j$, where $j=1,2,\dots,m$. Then, by definition, $\tau_j$'s are stopping times for the natural filtration of the reversed urn process, $\tau_1 \equiv 1$, and $\tau_1\leq \tau_2\leq\cdots\leq \tau_m\leq N$. Since the maximum is at least the average, $W_{nj}\geq j$, and hence
\begin{align}\label{eq:Tjnj}
	j\leq \tau_j\leq nj\text{ for all } 1\leq j\leq m.
\end{align}
Since $W_t=\abs{\{j\in[m]: \tau_j\leq t\}}=\sum_{j=1}^m \ind_{\{\tau_j\leq t\}}$, we obtain from~\eqref{eq:SW} the representation
\begin{align}\label{eq:STj}
	S_{n,m}=\sum_{j=1}^m\sum_{t=\tau_j}^N\frac1t=\sum_{j=1}^{m}  j\,(\cH_{\tau_{j+1}-1}-\cH_{\tau_j-1})= m\cH_N-\sum_{j=1}^m \cH_{\tau_j-1},
\end{align}
where $\tau_{m+1}:=N+1$.
Thus logarithmic hitting-time fluctuations govern the compensator.

% \begin{figure}[ht]
% 	\centering
% 	\setlength{\unitlength}{0.7cm}
% 	\begin{picture}(11,5)(-0.6,-1.0)
% 		\put(0,0){\vector(1,0){9.7}}
% 		\put(0,0){\vector(0,1){3.7}}
% 		\put(9.9,-0.1){$t$}
% 		\put(-0.8,3.6){$W_t$}
% 		\put(-0.4,-0.1){$0$}
% 		\put(-0.4,0.9){$1$}\put(-0.4,1.9){$2$}\put(-0.4,2.9){$3$}
% 		\put(0,0){\line(1,0){1}}
% 		\put(1,0){\line(0,1){1}}
% 		\put(1,1){\line(1,0){2}}
% 		\put(3,1){\line(0,1){1}}
% 		\put(3,2){\line(1,0){3}}
% 		\put(6,2){\line(0,1){1}}
% 		\put(6,3){\line(1,0){3}}
% 		\put(1,1){\circle*{0.12}}\put(2,1){\circle*{0.12}}
% 		\put(3,2){\circle*{0.12}}\put(4,2){\circle*{0.12}}\put(5,2){\circle*{0.12}}
% 		\put(6,3){\circle*{0.12}}\put(7,3){\circle*{0.12}}\put(8,3){\circle*{0.12}}\put(9,3){\circle*{0.12}}
% 		\put(1,-0.55){\makebox(0,0){$\tau_1=1$}}
% 		\put(3,-0.55){\makebox(0,0){$\tau_2=3$}}
% 		\put(6,-0.55){\makebox(0,0){$\tau_3=6$}}
% 		\put(9,-0.55){\makebox(0,0){$N=9$}}
% 	\end{picture}
% 	\caption{A reversed deck with $n=m=3$.
% 		Its types are $A,B,A,C,B,A,C,B,C$ in reversed order.
% 		The maximum increases at $1,3,6$.
% 		Thus $W_t=\sum_{j=1}^{3}\ind_{\{\tau_j\leq t\}}$ at integer times.}
% \label{fig:reversal}
% \end{figure}

\subsection{A coupling ordered by capacity}\label{ssec:capacitycoupling}
For decreasing vectors $x,y\in\mathbb Z_{\geq0}^n$ with the same sum, write
$x\preceq_{\mathrm{maj}}y$ if 
\[
\sum_{i=1}^kx_i\leq\sum_{i=1}^ky_i
\text{ for every }k.
\]

\begin{proof}[Proof of Proposition~\ref{prop:capacitycoupling}]
	Take a single sequence of independent $U_t\sim\mathrm{Unif}(0,1)$, $t\geq1$.
	For every capacity $c$, start all bins at zero.
	At time $t<nc$, sort the bins by decreasing current height, retaining their labels and breaking height ties by increasing label.
	If their sorted heights are $h_1\geq\cdots\geq h_n$, set
	\[
		F_c(h;r)=\frac{rc-\sum_{i=1}^r h_i}{nc-t},\quad 0\leq r\leq n,
	\]
	with $F_c(h;0)=0$. Increment the bin of rank
	\[
		I_c(t+1)=\min\{r\in[n]:U_{t+1}\leq F_c(h;r)\},
	\]
	and sort again.
	Each bin of height $h_i$ is chosen with probability
	$(c-h_i)/(nc-t)$, so retaining labels gives exactly the required urn-process marginal.

	We show that one common-uniform step preserves majorization for arbitrary capacities $c\leq d$.
	Suppose the sorted states $x,y$ at time $t<nc$ satisfy $x\preceq_{\mathrm{maj}}y$.
	Write $X_k=\sum_{i\leq k}x_i$ and $Y_k=\sum_{i\leq k}y_i$.
	After one increment and reordering, every partial sum increases by either zero or one.
	Thus, because the heights are integers, only a tight partial sum $X_k=Y_k$ could violate the order.

	Fix such a $k<n$, and put
	$r_x=\max\{i:x_i=x_k\}$ and $r_y=\max\{i:y_i=y_k\}$.
	The top-$k$ sum increases precisely when the selected rank is at most $r_x$ or $r_y$, respectively.
	We claim that $r_x\leq r_y$ and $X_{r_x}=Y_{r_x}$.
	If $r_x=k$, these assertions follow immediately from tightness and $r_y\geq k$.
	If $r_x>k$, then $x_{k+1}=x_k$.
	Tightness at $k$ and majorization at $k-1$ imply $x_k\geq y_k$.
	Majorization at $k+1$ gives
	$x_k=x_{k+1}\leq y_{k+1}\leq y_k\leq x_k$.
	Hence all these heights are equal, and the same argument propagates through the tie block of $x$ up to $r_x$.
	This proves the claim.
	The case $k=n$ needs no comparison, since both totals increase by one.

	Define $r=r_x$ and $H=X_r=Y_r$. Since $x$ is decreasing, $nH\geq rt$.
	For $a>t/n$, direct differentiation gives
	\[
		\frac{d}{da}\left(\frac{ra-H}{na-t}\right)
		=\frac{nH-rt}{(na-t)^2}\geq0.
	\]
	Consequently,
	\[
		F_c(x;r_x)\leq F_d(y;r_x)\leq F_d(y;r_y).
	\]
	Under the common uniform coupling, whenever the tight top-$k$ sum of $x$ increases, the corresponding sum of $y$ also increases.
	Every partial-sum inequality therefore survives the step.
	Starting from equal zero vectors proves the claimed order at all common times.
	Since the same uniforms serve every capacity, the construction is simultaneous for all integer capacities, and the almost-sure assertion holds for their countably many pairs and times.
	Taking $k=1$ gives the maximum comparison.
	Evaluating it at the smaller-capacity hitting time gives the hitting-time comparison.
	Finally, the backward representation~\eqref{eq:SW}, the maximum comparison on $t\leq nc$, and the positive extra terms on $nc<t\leq nd$ give the compensator comparison.
\end{proof}

\begin{cor}[Conditional binomials and multinomial comparison]\label{cor:capacitycomparisons}
	For fixed $n,t$, the maximum of independent $\Bin(m,1/2)$ variables conditionally on their sum being $t$ is stochastically nondecreasing in every admissible capacity $m$.
	Moreover, if $$(Z_1(t),Z_2(t),\ldots,Z_n(t))\sim\Mult(t;1/n,1/n,\ldots,1/n),
    $$ 
    then
    \begin{align}\label{eq:multcomparison}
		W_t^{(m)}\leq_{\mathrm{st}}\max_i Z_i(t),\quad t\leq nm.
	\end{align}
	In particular, for integers $k\geq1$,
	\begin{align}\label{eq:multmaxbound}
		\pr(W_t^{(m)}\geq k)\leq n\left(\frac{et}{nk}\right)^k.
	\end{align}
\end{cor}
\begin{proof}
	Conditioning independent binomials on their sum gives the mass function
	$\prod_i\binom m{x_i}/\binom{nm}t$, identical to the urn occupancy law.
	The first assertion follows from Proposition~\ref{prop:capacitycoupling}.
	For fixed $n,t$, the occupancy mass function at capacity $d$ can be written
	\[
		\frac{t!}{\prod_i x_i!}\frac{\prod_i(d)_{x_i}}{(nd)_t}
		\longrightarrow\frac{t!}{\prod_i x_i!}n^{-t}\text{ as }d\to\infty,
	\]
	where $\sum_i x_i=t$. The finite support and the capacity comparison prove~\eqref{eq:multcomparison}.
	Finally, each $Z_i(t)$ is $\Bin(t,1/n)$, so
	\[
		\pr(Z_i(t)\geq k)\leq\E\binom{Z_i(t)}k
		=\binom tk n^{-k}\leq\left(\frac{et}{nk}\right)^k.
	\]
	The union bound proves~\eqref{eq:multmaxbound}.
	For $k>t$, the probability is zero.
\end{proof}

\subsection{Hypergeometric estimates}\label{ssec:hyp}
For $1\leq j\leq m$ and $0\leq t\leq N$ let
\[
	p_j(t):=\pr\bigl(\widetilde X_1(t)\geq j\bigr)=\sum_{k\geq j}h_t(k),
\]
which is nondecreasing in $t$ (a sample of size $t+1$ contains a sample of size $t$), with $p_j(t)=0$ for $t<j$ and $p_j(N)=1$. We collect the facts about $p_j$ that we need.

\begin{lem}\label{lem:hyp}
	Fix $1\leq j\leq m$.
	\begin{enumeratea}
		\item\label{hyp:a} For all $0\leq t\leq N$,
		\[
			\pr(W_t\geq j)\leq n\,p_j(t)\quad\text{and}\quad \pr(W_t\leq j-1)\leq \bigl(1-p_j(t)\bigr)^n\leq e^{-n p_j(t)}.
		\]
		\item\label{hyp:b} For all $0\leq t\leq N-1$,
		\[
			p_j(t+1)-p_j(t)=\frac{j}{t+1}\,h_{t+1}(j).
		\]
		\item\label{hyp:c} For all $j\leq s\leq nj$,
		\[
			h_s(j)\geq \frac{1}{e^2\sqrt{j}}\; p_j(s).
		\]
	\end{enumeratea}
\end{lem}

\begin{proof}
	\ref{hyp:a} The first inequality is the union bound.
	For the second, use negative association of the multivariate hypergeometric vector.
	See Joag-Dev and Proschan~\cite[Section 3.1]{JDP83}.
	Negative association applies also to decreasing functions on disjoint coordinate sets.
	Iterating it for the indicators $\ind{\{\widetilde X_i(t)\leq x_i\}}$ gives

	\[
		\pr\bigl(\widetilde X_i(t)\leq x_i\ \forall i\in[n]\bigr)\leq \prod_{i=1}^n \pr\bigl(\widetilde X_i(t)\leq x_i\bigr).
	\] Taking $x_i=j-1$ for all $i$ gives $\pr(W_t\leq j-1)\leq (1-p_j(t))^n$, and $1-x\leq e^{-x}$ finishes the proof.

	\ref{hyp:b} Let $\xi_{t+1}$ be the indicator that the $(t+1)$-th ball goes into urn $1$. Then
	\[
		\{\widetilde X_1(t+1)\geq j\}\setminus\{\widetilde X_1(t)\geq j\}=\{\widetilde X_1(t+1)=j,\ \xi_{t+1}=1\}.
	\]
	If $h_{t+1}(j)=0$, both sides vanish. Otherwise, condition on $\widetilde X_1(t+1)=j$.
	The $j$ type-1 positions are exchangeable among these $t+1$ cards.
	Thus the final ball belongs to urn 1 with probability $j/(t+1)$. Hence $p_j(t+1)-p_j(t)=\pr(\widetilde X_1(t+1)=j)\cdot j/(t+1)$.

	\ref{hyp:c} Fix $s$ with $j\leq s\leq nj$. For $k\geq j$ with $h_s(k)>0$, that is, $j\leq k\leq \min(m,s)$, we have from~\eqref{eq:hyppmf}
	\begin{align}\label{eq:ratio}
		r_k(s):=\frac{h_s(k+1)}{h_s(k)}=\frac{(m-k)(s-k)}{(k+1)(N-m-s+k+1)},
	\end{align}
	where the denominator is positive because $s\leq nj$ and $k\geq j$ give $N-m-s+k+1\geq (n-1)(m-j)+1\geq 1$. (When $k=\min(m,s)$ the formula correctly gives $r_k(s)=0$.) We make two observations.
	First, $r_j(s)$ is nondecreasing in $s$, and at $s=nj$
	\[
		r_j(nj)=\frac{(m-j)\,j(n-1)}{(j+1)\bigl((n-1)(m-j)+1\bigr)}<1,
	\]
	so that $r_j(s)\leq 1$ for all $s\leq nj$. Second, for $j\leq k<\min(m,s)$, comparing the four factors in~\eqref{eq:ratio} at $k$ and $k+1$,
	\[
		\frac{r_{k+1}(s)}{r_k(s)}=\frac{m-k-1}{m-k}\cdot\frac{s-k-1}{s-k}\cdot\frac{k+1}{k+2}\cdot\frac{N-m-s+k+1}{N-m-s+k+2}\leq \frac{k+1}{k+2}.
	\]
	Iterating, $r_l(s)\leq r_j(s)\,\frac{j+1}{l+1}\leq \frac{j+1}{l+1}$ for all $j\le l\le \min(m,s)$ (the bound is trivial once $r_l(s)=0$). Therefore, for every $0\le i\le \min(m,s)-j$,
	\[
		h_s(j+i)=h_s(j)\prod_{l=j}^{j+i-1} r_l(s)\leq h_s(j)\prod_{l=j}^{j+i-1}\frac{j+1}{l+1}=h_s(j)\,\frac{(j+1)^i\, j!}{(j+i)!}.
	\]
    For $i>\min(m,s)-j$, $h_s(j+i)=0$, so the same upper bound holds. Summing over $i\geq 0$ and using $j!\leq e\sqrt{j}\,(j/e)^j$,
	\begin{align*}
		p_j(s) & \leq h_s(j)\, j!\,(j+1)^{-j}\sum_{i\geq 0}\frac{(j+1)^{j+i}}{(j+i)!} \\
		       & \leq h_s(j)\, j!\,(j+1)^{-j}e^{j+1}
		\leq h_s(j)\, e^2\sqrt{j}\,\bigl(\frac{j}{j+1}\bigr)^{j}\leq e^2\sqrt{j}\,h_s(j).
	\end{align*}
	This completes the proof.
\end{proof}

The key consequence of Lemma~\ref{lem:hyp}\ref{hyp:b},~\ref{hyp:c} is a lower bound on the \emph{elasticity} of $t\mapsto p_j(t)$. For $j-1\leq t\leq nj-1$, since $p_j(t+1)>0$,
\[
	\frac{p_j(t)}{p_j(t+1)}=1-\frac{j}{t+1}\cdot\frac{h_{t+1}(j)}{p_j(t+1)}\leq 1-\frac{\sqrt{j}}{e^2(t+1)}\leq \exp\bigl(-\frac{\alpha_j}{t+1}\bigr),\quad \alpha_j:=\frac{\sqrt{j}}{e^2}.
\]
Iterating, and using $\sum_{s=t+1}^{t'}1/s\geq \log\frac{t'+1}{t+1}$, we obtain for all $j\leq t<t'\leq nj$ (the case $t<j$ is immediate because $p_j(t)=0$)
\begin{align}\label{eq:elast}
	p_j(t)\leq p_j(t')\,\biggl(\frac{t+1}{t'+1}\biggr)^{\alpha_j}.
\end{align}

\subsection{Concentration of the hitting times: the universal bound}\label{ssec:univ}

\begin{lem}[Concentration from discrete elasticity]\label{lem:tailwindow}
	Let $U$ take values in $\{1,2,\ldots,R\}$, and let $g$ be nondecreasing on $\{0,1,\ldots,R\}$ with
	\[
		\pr(U\leq t)\leq g(t),\quad \pr(U>t)\leq e^{-g(t)}.
	\]
	Suppose $h:=\min\{t:g(t)\geq1\}$ exists, and for integers $0\leq t\leq t'\leq E$, with $h\leq E\leq R$,
	\[
		g(t)\leq g(t')\left(\frac{t+1}{t'+1}\right)^\beta,
		\quad 0<\beta\leq2h.
	\]
	If $E=R$, then $\E\log^2(U/h)\leq C/\beta^2$.
	If $E<R$, assume additionally that, for $0\leq t\leq t'\leq R$,
	\[
		g(t)\leq g(t')\left(\frac{t+1}{t'+1}\right)^\alpha,
		\quad 0<\alpha\leq E,
	\]
	and define $x_*:=\log(E/h)$. If $\beta x_*\geq6$ and $K:=e^{\beta x_*/2}$, then
	\begin{align}\label{eq:windowbound}
		\E\log^2(U/h)\leq
		C\left(\beta^{-2}+\alpha^{-2}e^{-K}
		 +(\alpha\beta)^{-1}e^{-K/2}\right).
	\end{align}
\end{lem}
\begin{proof}
	For the lower tail, take $t=\lfloor he^{-x}\rfloor$. If $t=0$, then $U\geq1>he^{-x}$ and the probability is zero.
	Otherwise $t\leq h-1$ and $g(h-1)<1$, so
	\[
		\pr(\log(U/h)<-x)\leq e^{-\beta x}(1+e^x/h)^\beta.
	\]
	For $0<x\leq2\log2$ the final factor is at most $e^8$, because $\beta/h\leq2$. For $x\geq2\log2$, whenever $t\geq1$ it is at most $2^\beta\leq e^{\beta x/2}$. Integrating $2x$ times these bounds gives $C/\beta^2$.
	For the upper tail, $t=\lfloor he^x\rfloor\leq E$ implies
	\[
		g(t)\geq g(h)\left(\frac{t+1}{h+1}\right)^\beta
		\geq e^{-2}e^{\beta x},
	\]
	where the equality-endpoint case is included.
	Thus the integral of $2x\exp(-e^{-2}e^{\beta x})$ is $C/\beta^2$. If $E=R$, the remaining upper tail is zero, proving the first claim.
	Otherwise $g(E)\geq e^{-2}e^{\beta x_*}$. For $x>x_*$ and $t<R$, the weaker elasticity and $\alpha/E\leq1$ give
	\[
		g(t)\geq g(E)e^{-1}e^{\alpha(x-x_*)}
		\geq K e^{\alpha(x-x_*)}.
	\]
	For $t\geq R$, the upper-tail probability is zero.
	Writing $y=x-x_*$ and using $e^{\alpha y}\geq1+\alpha y$, the remaining integral is at most
	\[
		e^{-K}\left(\frac{2}{K^2\alpha^2}+\frac{2x_*}{K\alpha}\right).
	\]
	With $u=\beta x_*/2$ we have
	$x_*e^{-K/2}=(2/\beta)u e^{-e^u/2}\leq C/\beta$.
	This proves~\eqref{eq:windowbound}.
\end{proof}

\begin{prop}\label{prop:logTj}
	There exists a universal constant $C_0<\infty$ such that for all $n\geq2$, $m\geq1$ and $1\leq j\leq m$ there is a deterministic $t_j=t_j(n,m)\in[j,nj]$ with
	\[
		\E\bigl(\log^2(\tau_j/t_j)\bigr)\leq \frac{C_0}{j}.
	\]
\end{prop}

\begin{proof}
	Define $g(t)=np_j(t)$ and define $t_j=\min\{t:g(t)\geq1\}$. Since $g(t)=0$ for $t<j$ and $g(nj)\geq\pr(W_{nj}\geq j)=1$, we have $j\leq t_j\leq nj$. Lemma~\ref{lem:hyp}\ref{hyp:a} gives
	\begin{align}\label{eq:tails}
		\pr(\tau_j\leq t)\leq g(t),\quad \pr(\tau_j>t)\leq e^{-g(t)}.
	\end{align}
	Apply the full-window case of Lemma~\ref{lem:tailwindow} with $U=\tau_j$, $h=t_j$, $E=R=nj$ and $\beta=\sqrt j/e^2$. The required elasticity is~\eqref{eq:elast}.
	The equal-endpoint case is an equality.
	Since $\beta\leq t_j$, the result follows with an absolute constant $C_0$.
\end{proof}

\begin{proof}[Proof of the universal bound~\eqref{eq:varSuniv}]
	We first note the elementary inequality
	\begin{align}\label{eq:Hlog}
		\abs{\cH_{a-1}-\cH_{b-1}}\leq 2\abs{\log(a/b)}\quad\text{for all integers } a,b\geq 1.
	\end{align}
	Indeed, for $a>b$,
	\[
		\cH_{a-1}-\cH_{b-1}=\sum_{k=b+1}^{a}\frac1k+\frac1b-\frac1a\leq \int_b^a\frac{dx}{x}+\frac{a-b}{ab}\leq \log\frac ab+\biggl(1-\frac ba\biggr)\leq 2\log\frac ab,
	\]
	using $1-y\leq -\log y$ for $y\in(0,1]$.
	Let $t_j$ be as in Proposition~\ref{prop:logTj}.
	By~\eqref{eq:Hlog} and Proposition~\ref{prop:logTj},
	\[
		\var(\cH_{\tau_j-1})\leq \E\bigl((\cH_{\tau_j-1}-\cH_{t_j-1})^2\bigr)\leq 4\,\E\bigl(\log^2(\tau_j/t_j)\bigr)\leq \frac{4C_0}{j}.
	\]
	By the representation~\eqref{eq:STj} and Minkowski's inequality,
	\[
		\norm{S_{n,m}-\E S_{n,m}}_2=\bigl\lVert\sum_{j=1}^m\bigl(\cH_{\tau_j-1}-\E(\cH_{\tau_j-1})\bigr)\bigr\rVert_2\leq \sum_{j=1}^m \sqrt{\frac{4C_0}{j}}\leq 4\sqrt{C_0\, m}.
	\]
	Hence $\var(S_{n,m})\leq 16 C_0\, m$, which is~\eqref{eq:varSuniv} with $C=16C_0$.
\end{proof}

\subsection{The large-deviation regime $j\leq \gl\log n$}\label{ssec:logregime}
Throughout this subsection $\gl>0$ is fixed.
The constants $\rho_0,\rho_1,\rho_2\in(0,1)$, $j_0$, $n_0$ and $C$ below depend only on $\gl$. The estimates are stated for indices $j\leq \gl\log n$ and hold for every $m\geq j$.
This uniformity in $m$ is used in Section~\ref{ssec:gaussregime}.

The key observation is that at $t_j$ the mean occupancy is at most a fixed fraction of $j$, uniformly in every capacity $m\geq j$. The probability ratios above $j$ are then bounded away from one, and the hazard $\phi_j$ is bounded below by a positive constant.
Thus the elasticity has order $j$ rather than $\sqrt j$. The location bound follows by comparison with capacity $j$, where a level-$j$ tail is simply the probability that all cards of one type have appeared.
The elasticity bound then follows from hypergeometric probability ratios.

\begin{lem}[Location of $t_j$]\label{lem:tjloc}
	Let $\rho_0:=e^{-1/(2\gl)}$. There exists $n_0=n_0(\gl)$ such that for all $n\geq n_0$, all $m\geq1$ and all $1\leq j\leq \min(m,\gl\log n)$, the hitting scale $t_j$ defined in Proposition~\ref{prop:logTj} satisfies
	\[
		t_j\leq \rho_0\, n j.
	\]
\end{lem}
\begin{proof}
	Let $t_0=\lfloor\rho_0nj\rfloor$, define $A=4e^2$ and define
	\[
		a_j=\rho_0e^{-A/\sqrt j},\quad u_j=\lfloor a_jnj\rfloor.
	\]
	For $n$ sufficiently large depending only on $\gl$, $j\leq u_j<t_0$ uniformly in $1\leq j\leq\gl\log n$.
	Consider first the process with capacity exactly $j$. A given bin is full at time $u_j$ with probability
	\[
		p_j^{(j)}(u_j)=\frac{(u_j)_j}{(nj)_j}
		\geq(a_j-1/n)^j\geq\tfrac12 a_j^j.
	\]
	The last inequality holds uniformly because $a_j\geq\rho_0e^{-A}>0$ and $j=O_\gl(\log n)$.
	Since $\rho_0^j\geq n^{-1/2}$,
	\[
		np_j^{(j)}(u_j)\geq\tfrac12\sqrt n\,
		\exp\{-A\sqrt{\gl\log n}\}\longrightarrow\infty
	\]
	uniformly in $j$.
	Negative association gives
	$\pr(W_{u_j}^{(j)}<j)\leq\exp\{-np_j^{(j)}(u_j)\}$.
	The capacity coupling therefore implies, for every $m\geq j$,
	\[
		\pr(W_{u_j}^{(m)}\geq j)\geq\tfrac12,
		\quad np_j^{(m)}(u_j)\geq\tfrac12,
	\]
	where the second inequality is the union bound.
	Apply the universal elasticity~\eqref{eq:elast} between $u_j$ and $t_0$, with $\beta=\sqrt j/e^2$.
	The rounding error is $O_\gl(1/n)$, uniformly in $j$.
	Since $j\leq\gl\log n$, for sufficiently large $n$,
	\[
		\log\frac{t_0+1}{u_j+1}\geq\frac{A}{2\sqrt j}.
	\]
	Consequently,
	\[
		np_j^{(m)}(t_0)\geq\tfrac12
		\left(\frac{t_0+1}{u_j+1}\right)^\beta
		\geq\tfrac12e^{A/(2e^2)}>1.
	\]
	By the definition of $t_j$, this gives $t_j\leq t_0\leq\rho_0nj$.
\end{proof}

\begin{lem}[Elasticity of order $j$]\label{lem:elastj}
	Let $\rho_1:=\sqrt{\rho_0}$ and $\rho_2:=(1+\rho_1)/2$. For $n\geq n_0(\gl)$, $m\geq1$, $1\leq j\leq m$ and all $j\leq s\leq \rho_1 nj$,
	\[
		\phi_j(s)=\frac{h_s(j)}{p_j(s)}\ \geq\ 1-\rho_2.
	\]
	Consequently, with $\beta_j:=(1-\rho_2)\,j$, for all $0\leq t<t'\leq \rho_1 nj$,
	\begin{align}\label{eq:elastj}
		p_j(t)\leq p_j(t')\bigl(\frac{t+1}{t'+1}\bigr)^{\beta_j}.
	\end{align}
\end{lem}
\begin{proof}
	Lemma~\ref{lem:hyp}\ref{hyp:c} gives $r_l(s)\leq r_j(s)$ for $l\geq j$.
	Whenever $r_j(s)<1$, summing the geometric bound yields
	\[
		p_j(s)\leq\frac{h_s(j)}{1-r_j(s)},\quad \phi_j(s)\geq1-r_j(s).
	\]
	By~\eqref{eq:ratio}, $r_j(s)$ is nondecreasing in $s$.
	For $s\leq\rho_1nj$,
	\begin{align*}
		r_j(s)
		 & \leq \frac{(m-j)\,\rho_1 nj}{(j+1)\bigl(nm-m-\rho_1 nj\bigr)}               \\
		 & \leq \frac{(m-j)\,\rho_1 j}{(j+1)(m-\rho_1 j)}\cdot\frac{n}{n-1/(1-\rho_1)}
		\leq \rho_1\cdot\frac{n}{n-1/(1-\rho_1)}\leq \rho_2,
	\end{align*}
	where we used $m\leq (m-\rho_1 j)/(1-\rho_1)$, then $(m-j)/(m-\rho_1 j)\leq 1, j/(j+1)\leq1$, and finally $n\geq n_0(\gl)\ge (1+\rho_1)/(1-\rho_1)^2$. This proves the first claim, and~\eqref{eq:elastj} follows exactly as~\eqref{eq:elast} was derived from Lemma~\ref{lem:hyp}\ref{hyp:b},\ref{hyp:c}, with $\alpha_j$ replaced by $j(1-\rho_2)$.
\end{proof}

\begin{prop}\label{prop:logTjlog}
	There exist $C_\gl$ and $n_0=n_0(\gl)$ such that for all $n\geq n_0$, all $m\geq 1$ and all $1\leq j\leq \min(m,\gl\log n)$,
	\[
		\E\bigl(\log^2(\tau_j/t_j)\bigr)\leq {C_\gl}/{j^2}.
	\]
\end{prop}
\begin{proof}
	Define $\beta=(1-\rho_2)j$, $\alpha=\sqrt j/e^2$, $g(t)=np_j(t)$, $R=nj$ and $E=\lfloor\rho_1nj\rfloor$. By Lemma~\ref{lem:tjloc}, $t_j\leq\rho_0nj$. Increasing $n_0(\gl)$ if necessary gives
	\[
		x_*:=\log(E/t_j)\geq\tfrac12\log(\rho_1/\rho_0)=:c_\gl>0
	\]
	uniformly in $j\geq1$. For the finitely many indices $j<j_0(\gl)$, Proposition~\ref{prop:logTj} gives the required bound after increasing $C_\gl$. Choose $j_0$ so that $\beta x_*\geq6$ for the remaining indices.
	The strong and weak elasticities are~\eqref{eq:elastj} and~\eqref{eq:elast}.
	Also $\beta\leq t_j$ and $\alpha\leq E$. Lemma~\ref{lem:tailwindow} therefore gives
	\[
		\E\log^2(\tau_j/t_j)\leq C\left(j^{-2}+j^{-1}e^{-K}
		+j^{-3/2}e^{-K/2}\right),\quad
		K\geq\exp\bigl((1-\rho_2)c_\gl j/2\bigr).
	\]
	The last two terms are bounded by $C_\gl/j^{2}$, proving the claim.
\end{proof}

\begin{proof}[Proof of the bound~\eqref{eq:varSlog}]
	Exactly as in the proof of~\eqref{eq:varSuniv}, by~\eqref{eq:Hlog}, Proposition~\ref{prop:logTjlog}, the representation~\eqref{eq:STj} and Minkowski's inequality,
	\begin{align*}
			\norm{S_{n,m}-\E S_{n,m}}_2
            &\leq \sum_{j=1}^m \sqrt{\var(\cH_{\tau_j-1})}\\
            &\leq \sum_{j=1}^m 2\sqrt{\E(\log^2(\tau_j/t_j))}\leq 2\sqrt{C}\sum_{j=1}^m\frac1j\leq 2\sqrt C\,(1+\log m)
	\end{align*}
	for $n\geq n_0(\gl)$ and $m\leq\gl\log n$. Squaring gives~\eqref{eq:varSlog} with $C_\gl=4C$.
\end{proof}

\begin{rem}
	Proposition~\ref{prop:logTjlog} is uniform in the capacity $m\geq j$, not only in decks with $m\leq\gl\log n$. For fixed $j\geq2$, Proposition~\ref{prop:exp} shows that the $j^{-2}$ scale is attained in the fixed-level limit.
	It does not establish matching lower bounds when $j$ grows with $n$.
\end{rem}

\subsection{The Gaussian regime $j>2\log n$}\label{ssec:gaussregime}
We now treat the $j>2\log n$ case. In the interior, occupancy deviations of order $\sqrt{j\log n}$ correspond to order $\sqrt{\log n}$ standard deviations.
Near capacity, the variance decreases, and the appropriate deviation depends on $m-j$. The proof below uses separate interior and saturation estimates to obtain a common elasticity lower bound of order $\sqrt{j\log n}$ for $j<m$. The terminal level $j=m$ is treated separately.
Throughout this subsection, $C$ and $n_0$ are absolute.
No attempt is made to optimize their values.

\subsubsection{A Stirling lower bound for the hypergeometric p.m.f.}
Let $H(y):=-y\log y-(1-y)\log(1-y)$, $y\in[0,1]$, and for $0\leq k\leq \min(m,t)$ with $t-k\leq N-m$, define
\begin{align}\label{eq:Lambda}
	\Lambda_t(k):=N H\biggl(\frac tN\biggr)-mH\biggl(\frac km\biggr)-(N-m)H\biggl(\frac{t-k}{N-m}\biggr).
\end{align}
Stirling's formula, $\sqrt{2\pi a}(a/e)^ae^{1/(12a+1)}\leq a!\leq \sqrt{2\pi a}(a/e)^ae^{1/(12a)}$, gives for integers $1\leq b\leq a-1$
\begin{align}\label{eq:binomrobbins}
	\frac{e^{aH(b/a)}}{3\sqrt{b(1-b/a)}}\ \leq\ \binom ab\ \leq\ \frac{e^{aH(b/a)}}{2\sqrt{b(1-b/a)}},
\end{align}
since $\binom ab=\sqrt{\frac{a}{2\pi b(a-b)}}\,e^{aH(b/a)}e^{\theta}$ with $-\frac16\leq\theta\leq\frac{1}{12}$, and $(2\pi)^{-1/2}e^{-1/6}>\frac13$, $(2\pi)^{-1/2}e^{1/12}<\frac12$.

\begin{lem}\label{lem:stirling}
	Let $n\geq 3$, $1\leq k\leq m-1$, and let $t$ be an integer with $1\leq t-k\leq N-m-1$ and $\mu:=t/n\leq k$. Then
	\begin{align}\label{eq:hstirling}
		h_t(k)\ \geq\ \frac{e^{-\Lambda_t(k)}}{7\sqrt{k(1-k/m)}}.
	\end{align}
	Moreover, if in addition $k-\mu\leq (n-1)\mu/2$, then
	\begin{align}
		\Lambda_t(k) & \leq (k-\mu)^2\biggl(\frac1\mu+\frac{1}{m-k}\biggr), \label{eq:LamF1}\text{ and } \\
		\Lambda_t(k) & \leq \frac{(k-\mu)^2}{\mu}+2(k-\mu). \label{eq:LamF2}
	\end{align}
\end{lem}
\begin{proof}
	Applying~\eqref{eq:binomrobbins} to the three binomial coefficients in $h_t(k)=\binom mk\binom{N-m}{t-k}\big/\binom Nt$ (lower bounds for the two in the numerator, the upper bound for the denominator) gives
	\[
		h_t(k)\ \geq\ \frac29\sqrt{\frac{t(1-t/N)}{k(1-k/m)\,(t-k)\bigl(1-\frac{t-k}{N-m}\bigr)}}\;e^{-\Lambda_t(k)}.
	\]
	Here $t/(t-k)\geq1$, and with $t=n\mu$, $N=nm$,
	\[
		\frac{1-t/N}{1-\frac{t-k}{N-m}}=\frac{(N-t)(N-m)}{N\,(N-m-t+k)}=\frac{(n-1)(m-\mu)}{n(m-\mu)-(m-k)}\ \geq\ 1-\frac1n\ \geq\ \frac12,
	\]
	using $m-k\geq0$. Since $\frac{2}{9\sqrt2}>\frac17$, this proves~\eqref{eq:hstirling}.

	For the bounds on $\Lambda_t$, regard $\Lambda_t(x)$ as a function of a real variable $x\in(0,m)$ with $0<t-x<N-m$. It is smooth there, with
	\begin{align*}
	\Lambda_t'(x) &=\log\frac{x}{m-x}-\log\frac{t-x}{N-m-t+x},\\
		\Lambda_t''(x)&=\frac1x+\frac{1}{m-x}+\frac{1}{t-x}+\frac{1}{N-m-t+x}>0.
	\end{align*}
	At $x=\mu=t/n$ we have $\frac{t-\mu}{N-m-t+\mu}=\frac{(n-1)\mu}{(n-1)(m-\mu)}=\frac{\mu}{m-\mu}$, so $\Lambda_t'(\mu)=0$, and 
    $$
    \Lambda_t(\mu)=NH(\mu/m)-mH(\mu/m)-(N-m)H(\mu/m)=0.
    $$ 
    Hence, for $k\geq \mu,$
    we have
	\begin{align}\label{eq:LamTaylor}
		\Lambda_t(k)=\int_\mu^k (k-x)\,\Lambda_t''(x)\,dx.
	\end{align}
	For $x\in[\mu,k]$, the four curvature terms satisfy
	\begin{align*}
		\frac1x     & \leq\frac1\mu,
		            & \frac1{m-x}           & \leq\frac1{m-k},        \\
		\frac1{t-x} & \leq\frac2{(n-1)\mu},
		            & \frac1{N-m-t+x}       & \leq\frac1{(n-1)(m-k)}.
	\end{align*}
	For the third bound, use $t-k=(n-1)\mu-(k-\mu)\geq(n-1)\mu/2$.
	For the fourth, use $x\geq\mu$ and $m-\mu\geq m-k$.
	Insert these bounds in~\eqref{eq:LamTaylor}.
	Since $\int_\mu^k(k-x)\,dx=(k-\mu)^2/2$, we obtain
	\[
		\Lambda_t(k)\leq \frac{(k-\mu)^2}{2}\biggl(1+\frac{2}{n-1}\biggr)\biggl(\frac1\mu+\frac1{m-k}\biggr)\leq (k-\mu)^2\biggl(\frac1\mu+\frac1{m-k}\biggr)
	\]
	for $n\geq3$, which is~\eqref{eq:LamF1}.
	For~\eqref{eq:LamF2}, treat the terms $1/(m-x)$ and $1/(N-m-t+x)$ separately.	We have
    \begin{align*}
	   \int_\mu^k\frac{k-x}{m-x}dx &\leq k-\mu\text{ as } \frac{k-x}{m-x}\leq1,\text{ and }\\ 
       \int_\mu^k\frac{k-x}{N-m-t+x}dx&\leq \frac{(k-\mu)^2}{2(n-1)(m-\mu)}\leq\frac{k-\mu}{2(n-1)}\text{ as }k-\mu\leq m-\mu.
    \end{align*}
    Hence
	\[
		\Lambda_t(k)\leq \frac{(k-\mu)^2}{2}\biggl(\frac1\mu+\frac{2}{(n-1)\mu}\biggr)+(k-\mu)\biggl(1+\frac{1}{2(n-1)}\biggr)\leq \frac{(k-\mu)^2}{\mu}+2(k-\mu)
	\]
	for $n\geq 3$.
\end{proof}

\subsubsection{Location of $t_j$ in the Gaussian regime}
For $2\log n\leq j\leq m-1$, define the target deviation
\begin{align}\label{eq:Dj}
	D_j:=\begin{cases}
		     \frac16\sqrt{j\log n},     & m\geq 3j,           \\[2pt]
		     \frac16\sqrt{(m-j)\log n}, & m<3j,\ m-j\geq 225, \\[2pt]
		     \frac14\log n,             & m-j< 225.
	     \end{cases}
\end{align}
Here, instead of $225$, any large enough constant will work.
In all three cases $D_j\leq j/6$.
For the first two, use $\log n\leq j/2$ and, in the second case, $m-j<2j$.
In the third case, $D_j\leq j/8$.

\begin{lem}[Location of $t_j$, Gaussian regime]\label{lem:tjlocG}
	There is an absolute constant $n_0$ such that for all $n\geq n_0$, $m\geq1$ and $2\log n\leq j\leq m-1$,
	\[
		t_j\leq n\,(j-D_j),\text{ \ie\ } d_j:=j-{t_j}/{n}\ \geq\ D_j.
	\]
\end{lem}
\begin{proof}
	Define $t=\lfloor n(j-D_j)\rfloor$, $\mu=t/n$ and $d=j-\mu$.
	Then $D_j\leq d\leq D_j+1/n$.
	We will show that $$
    np_j(t)\geq1.
    $$
	Since $D_j\leq j/6$, we have $\mu\geq5j/6-1/n\geq j/2$.
	The ratio estimates in the proof of Lemma~\ref{lem:hyp}.\ref{hyp:c} show that $h_t$ is nonincreasing above $j$.
	Hence, for every integer $b\geq1$ with $j+b-1\leq m-1$,
	\begin{align}\label{eq:blocksum}
		p_j(t)\ \geq\ \sum_{i=j}^{j+b-1}h_t(i)\ \geq\ b\,h_t(k),\quad k:=j+b-1.
	\end{align}

	We will apply Lemma~\ref{lem:stirling} to $h_t(k)$.
	Each block chosen below has $j\leq k\leq\min(2j-1,m-1)$.
	For $n\geq9$, its hypotheses follow from
	\begin{align*}
		t-k   & \geq n\mu-2j\geq(n/2-2)j\geq1, \\
		t-k   & \leq(n-1)j-nD_j\leq(n-1)m-1,   \\
		k-\mu & \leq2j\leq4\mu\leq(n-1)\mu/2.
	\end{align*}
	Here $m-j\geq1$ and $\mu\geq j/2$.
	Since $j\geq2\log n$, we can assume $j\geq200$ by increasing $n_0$.

	\smallskip\noindent\textbf{Case 1: $m\geq 3j$.} Take $b:=\lceil\sqrt j\rceil$, so $k\leq 2j-1\leq m-1$ and $m-k\geq m-2j\geq j\geq \mu$. By~\eqref{eq:hstirling}, $\frac{b}{7\sqrt{k(1-k/m)}}\geq \frac{\sqrt j}{7\sqrt{2j}}\geq \frac1{10}$, hence $p_j(t)\geq \frac1{10}e^{-\Lambda_t(k)}$. By~\eqref{eq:LamF1} and $m-k\geq\mu\geq j/2$,
	\[
		\Lambda_t(k)\leq \frac{2(k-\mu)^2}{\mu}\leq \frac{4(k-\mu)^2}{j},\quad k-\mu=d+b-1\leq \frac16\sqrt{j\log n}+\sqrt j+2\leq \frac14\sqrt{j\log n}
	\]
	for $n\geq n_0$ (as $\sqrt j+2\leq \frac{1}{12}\sqrt{j\log n}$ once $\log n\geq 144(1+2/\sqrt{200})^2$). Thus $\Lambda_t(k)\leq \frac14\log n$ and $np_j(t)\geq \frac{n^{3/4}}{10}\geq1$.

	\smallskip\noindent\textbf{Case 2: $m<3j$ and $m-j\geq 225$.} Take $b:=\lceil\sqrt{m-j}\rceil\leq \frac{m-j}{2}$, so $k\leq m-1$ and $m-k\geq (m-j)/{2}$. Since $k(1-k/m)\leq m-k\leq m-j$, $\frac{b}{7\sqrt{k(1-k/m)}}\geq\frac17$. Hence $p_j(t)\geq\frac17 e^{-\Lambda_t(k)}$. Now $m-j<2j\leq 4\mu$, so $\frac1\mu\leq\frac{4}{m-j}$, and~\eqref{eq:LamF1} gives
	\[
		\Lambda_t(k)\leq (k-\mu)^2\biggl(\frac{4}{m-j}+\frac{2}{m-j}\biggr)=\frac{6(k-\mu)^2}{m-j},
	\]
	where
	\[
		k-\mu\leq \frac16\sqrt{(m-j)\log n}+\sqrt{m-j}+2\leq \sqrt{m-j}\biggl(\frac{\sqrt{\log n}}{6}+\frac{17}{15}\biggr),
	\]
	using $2\leq \frac{2}{15}\sqrt{m-j}$. In particular, $\Lambda_t(k)\leq 6\Bigl(\frac{\sqrt{\log n}}{6}+\frac{17}{15}\Bigr)^2\leq \log n-\log 7$ for $n\geq n_0$, and $np_j(t)\geq1$.

	\smallskip\noindent\textbf{Case 3: $m-j\leq 224$.} Take $b=1$, $k=j$. Then $k(1-k/m)\leq m-j\leq 224$, so~\eqref{eq:hstirling} gives $p_j(t)\geq h_t(j)\geq \frac{1}{105}e^{-\Lambda_t(j)}$. Here $d\leq \frac14\log n+\frac1n$ and $\mu=j-d\geq 2\log n-\frac14\log n-1\geq \frac32\log n$, so~\eqref{eq:LamF2} gives
	\[
		\Lambda_t(j)\leq \frac{d^2}{\mu}+2d\leq \frac{(\frac14\log n+1)^2}{\frac32\log n}+\frac12\log n+2\leq \frac{3}{5}\log n\leq \log n-\log 105
	\]
	for $n\geq n_0$, and again $np_j(t)\geq1$.
\end{proof}

\subsubsection{Elasticity of order $\sqrt{j\log n}$}

\begin{lem}\label{lem:elastG}
	Let $n\geq 3$, $1\leq j\leq m-1$, and let $s$ be an integer with $j\leq s\leq nj$. Write $\mu:=s/n$ and $d:=j-\mu\geq0$. Then
	\begin{align}\label{eq:phiG}
		\phi_j(s)=\frac{h_s(j)}{p_j(s)}\ \geq\ 1-r_j(s)\ \geq\ \frac{m\,(d-1/n)}{(j+1)(m-\mu)}.
	\end{align}
	Consequently, for $n\geq n_0$, $m\geq1$, $2\log n\leq j\leq m-1$ and all integers $j\leq s\leq n(j-D_j/2)$,
	\begin{align}\label{eq:betaG}
		j\phi_j(s)\geq\beta_j:=\frac{m(D_j/2-1/n)}{2(m-j+D_j/2)},\text{ and } \frac{1}{60}\sqrt{j\log n}\leq \beta_j\leq\min\biggl(t_j,\frac{3j}{2}\biggr).
	\end{align}
	In particular, for all $0\leq t<t'\leq n(j-D_j/2)$,
	\begin{align}\label{eq:elastG}
		p_j(t)\leq p_j(t')\biggl(\frac{t+1}{t'+1}\biggr)^{\beta_j}.
	\end{align}
\end{lem}
\begin{proof}
	The first inequality in~\eqref{eq:phiG} was shown in the proof of Lemma~\ref{lem:elastj} (it only needs $r_j(s)<1$, which holds for $s\leq nj$). From~\eqref{eq:ratio}, with $s=n\mu$ and $N=nm$, a direct expansion gives
	\[
		1-r_j(s)=\frac{(j+1)(N-m-s+j+1)-(m-j)(s-j)}{(j+1)(N-m-s+j+1)}
		=\frac{nm(d+1)-n\mu-m+2j+1}{(j+1)\bigl(n(m-\mu)-(m-j-1)\bigr)}.
	\]
	The numerator is at least $nmd-m+n(m-\mu)\geq nmd-m$ (as $\mu\leq j\leq m$), and the denominator is at most $(j+1)n(m-\mu)$ because $m-j-1\geq0$. For $d\geq1/n$, these comparisons give the second inequality in~\eqref{eq:phiG}.
	If $d<1/n$, its right-hand side is negative whereas $1-r_j(s)>0$, so the inequality also holds.

	The function $d\mapsto m(d-1/n)/(m-j+d)$ is nondecreasing.
	In the stated window, $d\geq D_j/2$.
	Together with $j/(j+1)\geq1/2$, this proves $j\phi_j(s)\geq\beta_j$.

	For the upper bounds, first note $\beta_j\leq m/2$.
	If $j\geq m/2$, then $\beta_j\leq j\leq t_j$.
	If $j<m/2$, then $m-j+D_j/2>m/2$ and $\beta_j<D_j/2\leq j\leq t_j$.
	In either case, $\beta_j\leq3j/2$.

	For the lower bound, use the three cases in~\eqref{eq:Dj}.
	For sufficiently large $n$, each gives $D_j\geq(5/2)\sqrt{\log n}$.
	In the middle case, this follows from $m-j\geq225$.
	Thus $D_j/2-1/n\geq2D_j/5$ after increasing $n_0$.
    \begin{enumerate}[label={Case~\arabic*.},leftmargin=5em]
        \item ($m\geq3j$): $m-j+D_j/2\leq m$, so $\beta_j\geq \frac{D_j/2-1/n}{2}\geq \frac{D_j}{5}=\frac{\sqrt{j\log n}}{30}$.

    	\item ($m<3j$, $m-j\geq225$):
	If $D_j/2\geq m-j$, then
	\[
		\beta_j\geq\frac{m(D_j/2-1/n)}{2D_j}\geq\frac m5.
	\]
	Otherwise,
	\[
		\beta_j\geq\frac{m(D_j/2-1/n)}{4(m-j)}
		\geq\frac{mD_j}{10(m-j)}
		=\frac m{60}\sqrt{\frac{\log n}{m-j}}
		\geq\frac{\sqrt{m\log n}}{60}.
	\]
	Both bounds imply $\beta_j\geq\sqrt{j\log n}/60$.
	Here $m\geq j\geq2\log n$.

    	\item ($m-j\leq 224$): for $n\geq n_0$ we have $D_j/2=\frac18\log n\geq 224\geq m-j$, so as in the first sub-case of Case 2, $\beta_j\geq \frac m5\geq\frac15\sqrt{j\log n}$.
    \end{enumerate}

	Finally, Lemma~\ref{lem:hyp}\ref{hyp:b} gives
	\[
		\frac{p_j(t)}{p_j(t+1)}
		=1-\frac{j\phi_j(t+1)}{t+1}
		\leq\exp\left(-\frac{\beta_j}{t+1}\right)
	\]
	for $j-1\leq t\leq n(j-D_j/2)-1$.
	Iteration proves~\eqref{eq:elastG}.
\end{proof}

\subsubsection{Concentration of $\tau_j$ and proof of Theorem~\ref{thm:var}}

\begin{prop}\label{prop:logTjG}
	There exist absolute constants $C$ and $n_0$ such that for all $n\geq n_0$, $m\geq 1$ and $2\log n\leq j\leq m-1$,
	\[
		\E\bigl(\log^2(\tau_j/t_j)\bigr)\leq \frac{C}{j\log n}.
	\]
	Moreover, for all $n\geq2$ and $m\geq1$, $\E\bigl(\log^2(\tau_m/t_m)\bigr)\leq C/m^2$.
\end{prop}
\begin{proof}
	We follow the proof of Proposition~\ref{prop:logTjlog}, with $\beta:=\beta_j$ from Lemma~\ref{lem:elastG}, $\alpha:=\sqrt j/e^2$, $g(t)=np_j(t)$, and the window end
	\[
		E:=\lfloor n(j-D_j/2)\rfloor,\quad x_*:=\log\frac{E}{t_j}.
	\]
	By Lemma~\ref{lem:tjlocG}, $t_j\leq n(j-D_j)$, hence
	\[
		\frac{E}{t_j}\geq \frac{j-D_j/2-1/n}{j-D_j}=1+\frac{D_j/2-1/n}{j-D_j}\geq 1+\frac{D_j}{3j},\quad x_*\geq \log\biggl(1+\frac{D_j}{3j}\biggr)\geq \frac{D_j}{6j},
	\]
	using $D_j\leq j$ and $\log(1+y)\geq y/2$ for $y\in[0,1]$.
	The lower bounds for $\beta$ give explicit window products.
	In Case 1, $\beta x_*\geq(\log n)/1080$.
	In Case 2 with $D_j/2\geq m-j$, the product is at least $D_j/30\geq\sqrt{\log n}/12$.
	In the other subcase, it is at least $D_j^2/\{60(m-j)\}=(\log n)/2160$.
	In Case 3, it is at least $D_j/30=(\log n)/120$.
	Hence
	\begin{align}\label{eq:betaxstar}
		\beta x_*\ \geq\ \kappa_n:=\min\left\{\frac{\log n}{2160},\ \frac{\sqrt{\log n}}{12}\right\}\xrightarrow[n\to\infty]{}\infty.
	\end{align}

	For $n\geq n_0$,~\eqref{eq:betaxstar} ensures $\beta x_*\geq6$. Apply Lemma~\ref{lem:tailwindow} with $U=\tau_j$, $h=t_j$, $R=nj$, and the elasticities~\eqref{eq:elastG} and~\eqref{eq:elast}.
	Its hypotheses follow from $\beta\leq t_j$ and $\alpha\leq E$. Hence
	\[
		\E\log^2(\tau_j/t_j)\leq
		C\left(\frac1{\beta^2}+\frac{e^{-K}}{\alpha^2}
		+\frac{e^{-K/2}}{\alpha\beta}\right),\quad K\geq e^{\kappa_n/2}.
	\]
	By~\eqref{eq:betaG}, $\beta^2\geq j\log n/3600$. Also $\alpha^2=j/e^4$ and, for sufficiently large $n$, $\beta\geq\sqrt j$. Since $e^{-K/2}\leq1/\log n$ for large $n$, all three terms are bounded by $C/(j\log n)$.

	For $j=m$, $p_m(t)=h_t(m)$ and $\phi_m(s)=1$. Lemma~\ref{lem:hyp}\ref{hyp:b} yields
	\[
		\frac{p_m(t)}{p_m(t+1)}=1-\frac m{t+1}\leq e^{-m/(t+1)}
	\]
	for $m-1\le t\le N-1$. Thus the full-window elasticity has exponent $m$ up to $N$. The first part of Lemma~\ref{lem:tailwindow}, with $h=t_m$, $E=R=N$ and $\beta=m\leq t_m$, proves $\E\log^2(\tau_m/t_m)\leq C/m^2$.
\end{proof}

\subsection{Proof of Theorem~\ref{thm:var}}
	Choose $n_0$ large enough for Propositions~\ref{prop:logTjlog}, with $\gl=2$, and~\ref{prop:logTjG}.
	Let $n\geq n_0$ and $m\geq1$.
	Equation~\eqref{eq:Hlog} gives
	$\sqrt{\var(\cH_{\tau_j-1})}\leq2\sqrt{\E\log^2(\tau_j/t_j)}$ for every $j$.
	Use~\eqref{eq:STj} and Minkowski's inequality to obtain
	\begin{align*}
		\norm{S_{n,m}-\E S_{n,m}}_2
		 & \leq \sum_{1\leq j\leq \min(m,2\log n)}\frac{2\sqrt{C_2}}{j}
		+\sum_{2\log n< j\leq m-1}\frac{2\sqrt{C}}{\sqrt{j\log n}}+\ind_{\{m>2\log n\}}\frac{2\sqrt C}{m} \\
		 & \leq 2\sqrt{C_2}\,\bigl(1+\log(2\log n)\bigr)+\frac{4\sqrt{C m}}{\sqrt{\log n}}+2\sqrt C,
	\end{align*}

	The first sum uses Proposition~\ref{prop:logTjlog} with $\gl=2$.
	The second uses the Gaussian estimate for $j<m$.
	The last term uses the terminal-level estimate when $m>2\log n$.
	Squaring proves the theorem. \qed

\section{Hitting Times at Fixed Levels and Fixed Numbers of Types}\label{sec:hitting}
We identify the hitting-time law for each fixed level $j$.
The limit shows that the $j^{-2}$ variance scale is attained for fixed $j\geq2$.
It also explains the constants in the fixed-$m$ mean asymptotics.
These results are separate from the proof of the uniform score CLT.
We prove the fixed-level hitting-time limit stated in Proposition~\ref{prop:exp} first.

\subsection{Proof of Proposition~\ref{prop:exp}}
	Fix $j\geq2$, $y>0$ and an integer $r\geq1$. For sufficiently large $n$ we have $n\geq r$ and $t\geq rj$ below.
	Define
	\[
		\kappa=(m)_j/m^j,\quad
		b=(j!n^{j-1}/\kappa)^{1/j},\quad t=\lfloor by^{1/j}\rfloor,\quad
		Z=\sum_{i=1}^n\ind_{\{\widetilde X_i(t)\geq j\}}.
	\]
	Here $t\asymp n^{1-1/j}$ uniformly in $m\geq j$, for fixed $y>0$. For nonnegative integers $a_1,\ldots,a_r$ with $\sum_i a_i\leq N$, sampling without replacement gives
	\[
		\E\prod_{i=1}^r\biggl(\widetilde X_i(t)\biggr)_{a_i}
		=\frac{(t)_{\sum a_i}\prod_i(m)_{a_i}}{(N)_{\sum a_i}}.
	\]
	For all $r$ selected urns to contain exactly $j$ balls, the probability is
	\[
		\frac{(t)_{rj}(m)_j^r}{(j!)^r(N)_{rj}}
		\frac{(N-rm)_{t-rj}}{(N-rj)_{t-rj}}.
	\]
	The final ratio is the probability that, after the prescribed $rj$ selected balls, the remaining sample avoids all remaining balls of those $r$ types.
	A union bound bounds its distance from one by
	\[
		\frac{(t-rj)r(m-j)}{N-rj}=O_{r,j}(t/n).
	\]
	The formula also covers $m=j$, when the ratio is one.
	It follows uniformly in $m\geq j$ that the exact-$j$ probability is
	\[
		\left(\frac{\kappa t^j}{j!n^j}\right)^r
		\{1+O_{r,j}(t/n+1/t+1/n)\}.
	\]
	For the excess probability, use a union bound over the urns containing at least $j+1$ balls.
	Apply the factorial identity with one $a_i=j+1$ and all other $a_i=j$.
	Markov's inequality gives
	\[
		\pr(\widetilde X_i(t)\geq j\ (1\leq i\leq r),\ \max_{i\leq r}\widetilde X_i(t)>j)
		=O_{r,j}((t/n)^{rj+1}).
	\]
	Since $\kappa\geq j!/j^j$, the excess is relatively $O_{r,j}(t/n)$. Thus
	\[
		\E(Z)_r=(n)_r\pr(\widetilde X_i(t)\geq j\ (1\leq i\leq r))\longrightarrow y^r.
	\]
	The factorial-moment criterion for Poisson convergence~\cite{BHJ92} gives $Z\Rightarrow\mathrm{Poi}(y)$, and therefore $\pr(\tau_j>t)\to e^{-y}$. All estimates are uniform in $m\geq j$. Equivalently, the argument applies to every sequence $m=m_n\geq j$. This proves the uniform exponential limit.

	The one-urn computation also gives $np_j(\lfloor bv\rfloor)\to v^j$, uniformly in $m\geq j$.
	Hence $t_j/b\to1$.
	Apply the full-window tail estimates in Lemma~\ref{lem:tailwindow} with $\beta=\sqrt j/e^2$.
	They bound the tails of $\log(\tau_j/t_j)$ by integrable exponential tails.
	Every fixed polynomial moment is therefore uniformly integrable.
	Since $\log(t_j/b)\to0$, these estimates establish uniform integrability of $\log^2(\tau_j/b)$, and justify the first and second logarithmic moment limits stated below.\qed

\smallskip Let $E\sim\mathrm{Exp}(1)$, and recall that $\E(\log E)=-\gamma$ and $\var(\log E)=\pi^2/6$, where $\gamma\approx 0.57721$ is the Euler--Mascheroni constant.
By Proposition~\ref{prop:exp},
\[
	j\log \tau_j-(j-1)\log n-\log\frac{j!}{\kappa_{m,j}}\implies \log E,
\]
and the uniform integrability established in its proof gives convergence of the first two moments.
Hence, for fixed $j\geq2$ and $m\geq j$, as $n\to\infty$,
\begin{align}\label{eq:logTjasymp}
	\E(\log \tau_j)=\bigl(1-\frac1j\bigr)\log n+\frac1j\log\frac{m^j}{\binom mj}-\frac{\gamma}{j}+o(1),
	\quad
	\var(\log \tau_j)\to\frac{\pi^2}{6j^2},
\end{align}
where we used $j!/\kappa_{m,j}=m^j/\binom mj$.

\begin{cor}[Compensator limit for fixed multiplicity]\label{cor:fixedmS}
	Fix $m\geq1$.
	Let $E_2,\ldots,E_m$ be independent exponential random variables of mean one.
	Then, as $n\to\infty$,
	\[
		S_{n,m}-\E S_{n,m}
		\implies-\sum_{j=2}^{m}\frac{\log E_j+\gamma}{j},
		\quad
		\var(S_{n,m})\to\frac{\pi^2}{6}\sum_{j=2}^{m}\frac1{j^2}.
	\]
	Both sums are zero when $m=1$.
\end{cor}
\begin{proof}
	For $m=1$, $W_t=1$ for $t\geq1$ and $S_{n,1}=\cH_n$.
	Now suppose $m\geq2$.
	Set $b_j=(j!n^{j-1}/\kappa_{m,j})^{1/j}$.
	He and Ottolini~\cite[Theorem 4.1]{He_2023} prove a joint Poisson approximation for the birthday counts.
	Their last time without a $j$-fold repeat is $\tau_j-1$ in our notation.
	At $t_j'=\lfloor b_j y_j^{1/j}\rfloor$, their Poisson parameter tends to $y_j$.
	For fixed positive $y_2,\ldots,y_m$, the times are eventually ordered because their powers of $n$ strictly increase.
	Their joint survival estimate therefore gives
	\[
		\pr(\tau_j>t_j'\text{ for all }2\leq j\leq m)
		\longrightarrow\exp\left(-\sum_{j=2}^{m}y_j\right).
	\]
	Applying the same result to subsets of indices gives the joint distributional limit
	\[
		\left((\tau_j/b_j)^j\right)_{j=2}^{m}
		\implies(E_j)_{j=2}^{m}.
	\]
	The full-window bounds in Lemma~\ref{lem:tailwindow} give bounded fourth moments of $\log(\tau_j/b_j)$.
	Thus the logarithms and their pairwise products are uniformly integrable.
	The first two joint moments consequently converge.
	The harmonic-number error tends to zero in $L^2$, since it is bounded by $C/\tau_j$ and $\tau_j\to\infty$ in probability.
	Use~\eqref{eq:STj} and center each term.
	The joint limit yields the displayed law, and independence gives the variance constant.
\end{proof}

\begin{rem}[Consistency with the mean for fixed $m$]

	The expansion~\eqref{eq:logTjasymp} recovers the fixed-$m$ mean asymptotics of~\cite{He_2023}.
	For $j\geq2$, $\tau_j\to\infty$ in probability.
	The deterministic error in $\cH_{a-1}=\log a+\gamma+O(1/a)$ is bounded for $a\geq1$.
	Hence $\E\cH_{\tau_j-1}=\E\log\tau_j+\gamma+o(1)$ since the error is uniformly bounded and tends to zero in probability.
	The level $j=1$ contributes $\cH_0=0$.
	Inserting~\eqref{eq:logTjasymp} in the representation~\eqref{eq:STj} and using $\cH_N=\log n+\log m+\gamma+o(1)$ gives
	\begin{align*}
		\E S_{n,m} & =m\,\cH_N-\sum_{j=2}^m\bigl(\E(\log \tau_j)+\gamma\bigr)+o(1)
		=\cH_m(\log n+\gamma)+\sum_{j=1}^{m-1}\frac1j\log\binom mj+o(1),
	\end{align*}
	which agrees with $\cH_m\cH_n+\sum_{j=1}^{m-1}\frac1j\log\binom mj$ up to $o(1)$.
\end{rem}

% \begin{rem}[Occupancy variance and bridge scaling]
% 	The exact one-coordinate variance is
% 	\[
% 		\sigma_t^2:=\var(\widetilde X_i(t))=
% 		\frac tn\left(1-\frac tN\right)\frac{N-m}{N-1},\quad N>1.
% 	\]
% 	As $n\to\infty$, the final factor tends to one uniformly in $m$. For fixed $n$, it instead tends to $1-1/n$ as $m\to\infty$. The corresponding joint bridge limit is proved in Section~\ref{sec:bb}.
% 	An interior Gaussian extreme-value heuristic suggests a maximum fluctuation scale $\sigma_t/\sqrt{\log n}$, but requires separate treatment of lattice and saturation effects.
% \end{rem}

\subsection{Proof of Proposition~\ref{prop:fixedn}}
	Give every card an independent uniform $[0,1]$ mark and read the reversed deck in increasing order of the marks.
	For type $i$, let $F_{i,m}$ be its empirical distribution function.
	Define $G_m=n^{-1}\sum_iF_{i,m}$.
	The independent empirical-process invariance principles give jointly
	\[
		\sqrt m(F_{i,m}-\mathrm{id})\implies B_i,
		\quad \sqrt m(G_m-\mathrm{id})\implies\overline B.
	\]
	Let $q_m(u)$ be the pooled order statistic of rank $\lfloor uN\rfloor$.
	Write $r_m(u)=\lfloor uN\rfloor/N$.
	Almost surely, $G_m(q_m(u))=r_m(u)$ whenever $1\leq\lfloor uN\rfloor\leq N$.
	Uniform convergence of $G_m$ to the identity gives $q_m(u)\to u$ uniformly on compact subintervals.
	The exact identity
	\[
		\sqrt m(q_m(u)-u)
		=-\sqrt m(G_m(q_m(u))-q_m(u))
		+\sqrt m(r_m(u)-u)
	\]
	and stochastic equicontinuity yield the inverse-process limit.
	The last term is at most $1/(n\sqrt m)$ in absolute value.
	Apply the same equicontinuity to each $F_{i,m}$.
	Since $mF_{i,m}(q_m(u))=\widetilde X_i(\lfloor uN\rfloor)$, this gives
	\[
		\sqrt m(q_m(u)-u)\implies-\overline B(u),\quad
		\frac{\widetilde X_i(\lfloor uN\rfloor)-\lfloor uN\rfloor/n}{\sqrt m}
		\implies B_i(u)-\overline B(u).
	\]
	The rank rounding error is $O(m^{-1/2})$ because $n$ is fixed.
	This proves the joint bridge limit on truncated intervals.
	See~\cite{Bill99,vdVW96} for the invariance principle and stochastic equicontinuity.

	Define $A_t=W_t-t/n\geq0$. Identity~\eqref{eq:SW} gives
	\[
		\frac{S_{n,m}-m}{\sqrt m}=\frac1{\sqrt m}\sum_{t=1}^N\frac{A_t}{t}.
	\]
	Since a maximum is bounded by the Euclidean norm, we have
	$
		\norm{A_t}_2^2\leq\sum_{i=1}^n\var(\widetilde X_i(t))
		\leq t(1-t/N).
	$
	Consequently, for $K=\lfloor\delta N\rfloor$,
	$
		\norm{\frac1{\sqrt m}\sum_{t=1}^K {A_t}/t}_2
		\leq\frac1{\sqrt m}\sum_{t=1}^K t^{-1/2}
		\leq2\sqrt{n\delta}.
	$
	At the other endpoint, for $0<\delta<1/2$,
	$
		\norm{\frac1{\sqrt m}\sum_{t\geq(1-\delta)N} {A_t}/t}_2
		\leq C_n(\delta+N^{-1})^{3/2}.
	$
	This follows by writing $r=N-t$, using $A_N=0$, and summing
	$\sqrt{r/N}/\sqrt t\leq C\sqrt r/N$ before dividing by $\sqrt m$. The limiting Gaussian vector satisfies the analogous bounds since
	$\E\sum_i(B_i(u)-\overline B(u))^2=(n-1)u(1-u)$.
	On $[\delta,1-\delta]$, the bridge convergence, continuity of the maximum, and Riemann sums give convergence of the truncated integrals.
	The endpoint bounds and the converging-together argument prove the assertion.
	They also bound the whole expression in $L^2$, so its first moments are uniformly integrable.
	Finally,
	$
		\E\max_i(B_i(u)-\overline B(u))
		=\sqrt{u(1-u)}\,\cZ_n,
	$
	which integrates to $(\pi/2)\cZ_n$. The same norm bound proves that $\mathcal{B}_n$ exists in $L^2$.\qed

\begin{rem}[SDE representations]\label{rem:fixednSDE}
	The bridge functional admits a diffusion description.
	This gives a stochastic representation of its distribution.
	Define
	\[
		P_n=I_n-n^{-1}\boldsymbol1\boldsymbol1^{\mathsf T},\quad
		\mathsf H_n=\{x\in\dR^n:\sum x_i=0\},\quad g(x)=\max_i x_i.
	\]
	For standard $n$-dimensional Brownian motion $\mvW$, consider
	\begin{align}\label{eq:bridgeSDE}
		dX(u)& =P_nd\mvW(u)-\frac{X(u)}{1-u}\,du,
		      \ X(0) =0,\quad\text{ and }\quad
		dA(u) =\frac{g(X(u))}{u}\,du,
		      \  A(0)=0.
	\end{align}
	The linear SDE has the unique solution
	$
		X(u)=(1-u)\int_0^u {P_n d\mvW(v)}/(1-v).
	$
	Its covariance is $(\min(u,v)-uv)P_n$.
	Thus $X$ has the law of $(B_i-\overline B)_{i\leq n}$ and extends continuously to $X(1)=0$.
	The endpoint estimates above make the improper integral for $A$ well defined in $L^2$.
	Hence $\cB_n\stackrel d=A(1)$.
\end{rem}

\smallskip\noindent\textbf{Declaration on the use of AI.}
The mathematical ideas in the original manuscript were developed by the authors without the use of AI.
The authors used ChatGPT at the final revision stage to assist with proofreading, grammar checking, and reviewing the mathematical arguments and their presentation.  All AI suggestions incorporated into the manuscript were reviewed by the authors, who take full responsibility for the content of the paper.

\bibliographystyle{abbrvurl}
\bibliography{cg}
\end{document}

%% file: style.tex
\usepackage{mathrsfs,xfrac} 
\usepackage[colorlinks=true,linkcolor=blue,citecolor=blue,urlcolor=blue]{hyperref} 
\usepackage{amsmath,amssymb,amsthm,amsfonts,amsbsy,latexsym,dsfont,color,graphicx,enumitem}
\usepackage[foot]{amsaddr}
\usepackage{caption}
\usepackage{comment}

\usepackage[textsize=tiny]{todonotes}
\usepackage{regexpatch}
\makeatletter
\xpatchcmd{\@todo}{\setkeys{todonotes}{#1}}{\setkeys{todonotes}{inline,#1}}{}{}
\makeatother

\usepackage{enumitem}

\newenvironment{enumeratea}{\begin{enumerate}[label=\textup{(\alph*)},leftmargin=2em,itemsep=1ex]}{\end{enumerate}}

\newtheorem{thm}{Theorem}[section]
\newtheorem{lem}[thm]{Lemma}
\newtheorem{cor}[thm]{Corollary}
\newtheorem{prop}[thm]{Proposition}

\newtheorem{rem}{Remark}[section]
\renewcommand{\le}{\leqslant}  
\renewcommand{\ge}{\geqslant} 
\renewcommand{\leq}{\leqslant}  
\renewcommand{\geq}{\geqslant}

\newcommand{\ind}{\mathds{1}}
\newcommand{\eps}{\varepsilon}

\newcommand{\norm}[1]{\left\Vert#1\right\Vert}
\newcommand{\abs}[1]{\left\vert#1\right\vert}

\newcommand{\ie}{\emph{i.e.,}}

\def\qed{ \hfill $\blacksquare$}  
     \let\gl=\lambda

\newcommand{\cB}{\mathcal{B}}
\newcommand{\cF}{\mathcal{F}}
\newcommand{\cH}{\mathcal{H}}

\newcommand{\cO}{\mathcal{O}}

\newcommand{\cZ}{\mathcal{Z}}  
\newcommand{\vX}{\mathbf{X}}\newcommand{\vY}{\mathbf{Y}}

\newcommand{\vx}{\mathbf{x}}
 
\newcommand{\mvW}{\boldsymbol{W}}

\newcommand{\dR}{\mathds{R}}

\DeclareMathOperator{\E}{\mathds{E}}
\DeclareMathOperator{\pr}{\mathds{P}}

\DeclareMathOperator{\var}{Var}
\DeclareMathOperator{\cov}{Cov}
\DeclareMathOperator{\argmax}{argmax}

\DeclareMathOperator{\hypgeom}{Hypergeometric}

\DeclareMathOperator{\Ber}{Ber}
\DeclareMathOperator{\N}{N}
\DeclareMathOperator{\Mult}{Multinomial}

\DeclareMathOperator{\s}{\mathbf{score}}
\DeclareMathOperator{\Bin}{Bin}